\documentclass[12pt, reqno]{amsart}

\usepackage{graphicx}
\usepackage{subfigure}
\usepackage{eepic}
\usepackage{xcolor}
\selectcolormodel{gray}
\usepackage{tikz}
\usepackage{amsmath}
\usepackage{a4wide}
\usepackage[utf8]{inputenc}
\usepackage{amssymb}
\usepackage{amsopn}
\usepackage{epsfig}
\usepackage{amsfonts}
\usepackage{latexsym}
\usepackage{amsthm}
\usepackage{enumerate}
\usepackage[UKenglish]{babel}
\usepackage{verbatim}
\usepackage{color}

\newtheorem{thm}{Theorem}[section]
\newtheorem{lma}[thm]{Lemma}
\newtheorem{cor}[thm]{Corollary}
\newtheorem{defn}[thm]{Definition}

\newtheorem{prop}[thm]{Proposition}

\newtheorem{hyp}[thm]{Hypothesis}

\newcommand{\R}{\mathbb{R}}
\newcommand{\Q}{\mathbb{Q}}

\newcommand{\N}{\mathbb{N}}

\providecommand{\norm}[1]{\lVert#1\rVert}

\renewcommand{\geq}{\geqslant}
\renewcommand{\leq}{\leqslant}
\renewcommand{\epsilon}{\varepsilon}
\renewcommand{\H}{\text{H}}

\renewcommand{\l}{\text{loc}}
\renewcommand{\i}{\mathbf{i}}
\newcommand{\n}{\mathbf{n}}
\renewcommand{\j}{\mathbf{j}}

\renewcommand{\geq}{\geqslant}
\renewcommand{\leq}{\leqslant}
\renewcommand{\H}{\mathcal{H}}

\renewcommand{\i}{\mathbf{i}}
\providecommand{\I}{\mathcal{I}}
\providecommand{\cont}{C([0,1])}
\providecommand{\lip}{C^{0,1}([0,1])}
\renewcommand{\j}{\mathbf{j}}
\providecommand{\w}{\mathbf{w}}

\renewcommand{\l}{\mathcal{L}}
\providecommand{\p}{\mathbf{p}}
\providecommand{\q}{\mathbf{q}}
\providecommand{\i}{\mathbf{i}}
\providecommand{\m}{\mathcal{M}}
\providecommand{\norm}[1]{\lVert#1\rVert}

\providecommand{\one}{\mathbf{1}}

\renewcommand{\c}{\mathcal{C}_a}
\providecommand{\n}{n^{\ast}}

\begin{document}

\title[]{A new proof of the dimension gap for the Gauss map}

\author{Natalia Jurga}
\address{Natalia Jurga: Department of Mathematics, University of Surrey, Guildford, GU2 7XH, UK}
\email{N.Jurga@surrey.ac.uk }

\date{\today}

\subjclass[2010]{}

\begin{abstract}
In \cite{kpw}, Kifer, Peres and Weiss showed that the Bernoulli measures for the Gauss map $T(x)= \frac{1}{x} \mod 1$ satisfy a `dimension gap' meaning that for some $c>0$, $\sup_{\p} \dim \mu_{\p} <1-c$, where $\mu_{\p}$ denotes the (pushforward) Bernoulli measure for the countable probability vector $\p$. In this paper we propose a new proof of the dimension gap. By using tools from thermodynamic formalism we show that the problem reduces to obtaining uniform lower bounds on the asymptotic variance of a class of potentials.
\end{abstract}

\maketitle

\section{Introduction}\label{intro}

Let $x \in [0,1] \setminus \Q$. It is well known that there exists a sequence $\{i_n\}_{n \in \N}$ known as the \emph{continued fraction expansion} of $x$ that satisfies
$$x=\cfrac{1}{i_1+\cfrac{1}{i_2+\cfrac{1}{i_3+\ldots}}}.$$
Continued fractions are closely related to the \emph{Gauss map} which is defined as $T: [0,1] \setminus \Q \to [0,1] \setminus \Q$
$$T(x)=\frac{1}{x} \mod 1.$$
Let $\Sigma=\N^{\N}$, $\Sigma^{\ast}$ denote the set of all  words of finite length with entries in $\N$ and $\sigma: \Sigma \to \Sigma$ be the shift map given by $\sigma((i_n)_{n \in \N})=(i_{n+1})_{n \in \N}$. Often we will let $\i$ denote a point $\i=(i_n)_{n \in \N} \in \Sigma$. $T$ is `coded' by $(\Sigma, \sigma)$ meaning that $T \circ \Pi= \Pi \circ \sigma$ where the `coding map' $\Pi: \Sigma \to [0,1] \setminus \Q$ is given by
$$\Pi(\i)= \lim_{n \to \infty} T_{i_1}^{-1} \circ \ldots T_{i_n}^{-1}([0,1])= \cfrac{1}{i_1+\cfrac{1}{i_2+\cfrac{1}{i_3+\ldots}}}.$$

It is well known that $T$ has an absolutely continuous invariant probability measure $\mu_T$ given by
$$\mu_T(A)= \frac{1}{\log 2} \int_A \frac{1}{1+x} \textup{d}x.$$
By using the coding map $\Pi$, we can construct many more $T$-invariant measures, by `pushing forward' $\sigma$-invariant measures from $\Sigma$. In particular, if $m$ is a $\sigma$-invariant measure then $\mu=m \circ \Pi^{-1}$ is a $T$-invariant measure. In this paper we will be focused on pushforward \emph{Bernoulli measures}. Given a countable probability vector $\p=(p_n)_{n \in \N}$, let $m_{\p}$ denote the Bernoulli measure on $\Sigma$ which satisfies $m_{\p}([i_1\ldots i_n])=p_{i_1} \ldots p_{i_n}$, where $[ i_1 \ldots i_n]=\{\j \in \Sigma: j_1=i_1, \ldots ,j_n=i_n\}$ denotes the cylinder set for the word $i_1 \ldots i_n$. We define $\mu_{\p}=m_{\p} \circ \Pi^{-1}$ and we will also call this a Bernoulli measure. 

We will be interested in the \emph{Hausdorff dimension} of Bernoulli measures, where the Hausdorff dimension of a Borel probability measure $\mu$ is defined as
$$\dim \mu= \inf\{\dim A: \mu(A)=1\}$$
where $\dim A$ denotes the Hausdorff dimension of the set $A$. By the work of Walters \cite{walters}, $\mu_T$ is the unique absolutely continuous invariant probability measure for $T$ and realises the supremum
\begin{eqnarray}
h(\mu_T)- \int \log|T^{\prime}|\textup{d}\mu_T= \sup_{\mu \in \mathcal{M}(T)}\left\{ h(\mu)- \int \log|T^{\prime}|\textup{d}\mu: \int \log|T^{\prime}|\textup{d}\mu< \infty\right\}=0
\label{vp}
\end{eqnarray}
where $\mathcal{M}(T)$ denotes all $T$-invariant probability measures and $h(\cdot)$ denotes the measure-theoretic \emph{entropy}. As a direct consequence of (\ref{vp}) we deduce that for any $\p$ for which $h(\mu_{\p})< \infty$, 
\begin{eqnarray}
\dim \mu_{\p}= \frac{h(\mu_{\p})}{\chi(\mu_{\p})} < 1
\label{less1}
\end{eqnarray}
where the formula $\dim (\cdot)=\frac{h(\cdot)}{\chi(\cdot)}$ is known to hold for all finite entropy ergodic measures and $\chi(\mu_{\p})= \int \log|T^{\prime}| \textup{d}\mu_{\p}$ is known as the \emph{Lyapunov exponent} of $\mu_{\p}$. 

What is not clear from (\ref{less1}) is whether there is a `\emph{dimension gap}' at 1. We say that there is a dimension gap if there exists some $c>0$ for which
$$\sup_{\p \in \mathcal{P}} \dim \mu_{\p} \leq 1-c$$
where $\mathcal{P}$ denotes the simplex of all probability vectors. In this paper we will prove the following result.

\begin{thm} \label{main}
There exists $c>0$ such that 
$$\sup_{\p \in \mathcal{P}} \dim \mu_{\p} \leq 1-c.$$
\end{thm}

Theorem \ref{main} was already proved by Kifer, Peres and Weiss \cite{kpw} who showed that
\begin{eqnarray}
\sup_{\p \in \mathcal{P}} \dim \mu_{\p} \leq 1-10^{-7}.
\label{kpwgap}
\end{eqnarray}
We briefly sketch their proof. Given $\mathbf{w} \in \Sigma^{\ast}$ and $\delta>0$ let $\Gamma_{\mathbf{w}}^{\delta}$ be defined by 
$$\Gamma_{\mathbf{w}}^{\delta}=\left\{ x \in (0,1): \limsup_{n \to \infty} \left|\frac{1}{n}\sum_{i=0}^{n-1} \one_{\mathbf{w}}(T^ix)-\mu_T(\Pi(\mathbf{w}))\right|>\delta\right\}$$
which is the set of points whose orbits visit the interval $\Pi([\mathbf{w}])$ with an asymptotic frequency which differs by $\delta$ from the one prescribed by $\mu_T$. By using the ergodic theorem it is not difficult to show that for some $\delta_0>0$, $\dim \mu_{\p} \leq \max\{\dim \Gamma_{1}^{\delta_0}, \dim \Gamma_{11}^{\delta_0}\}$ for all $\p$. Also define $J_n(x)= \Pi([i_1\ldots i_n])$ if $ x \in \Pi([i_1\ldots i_n])$, that is, $J_n(x)$ is the `level $n$' projected cylinder that $x$ belongs to, let $|J_n(x)|$ denote the diameter of $J_n(x)$ and consider the set
\begin{eqnarray} \label{e1}
\mathcal{E}_{\lambda}= \bigcap_{j=1}^{\infty} \bigcup_{n=j}^{\infty} \left\{x \in (0,1): |J_n(x)| \leq \exp(-\lambda n)\right\}
\end{eqnarray}
which is the set of points whose orbits `frequently' visit a `small' neighbourhood of 0. Kifer, Peres and Weiss showed that for some $\lambda_0>0$, $\dim \mathcal{E}_{\lambda_0}<1$ which allowed them to reduce the problem down to finding an upper bound for the dimension of the set of points in $\Gamma_{1}^{\delta_0}$ and $\Gamma_{11}^{\delta_0}$ which \emph{don't} belong to $\mathcal{E}_{\lambda_0}$. They then showed that for any $\delta>0$,
$$\sup_{\w \in \Sigma^{\ast}} \dim (\Gamma_{\w}^{\delta}\setminus \mathcal{E}_{\lambda_0}) < 1$$
which completed the proof.

Another proof of Theorem \ref{main} was given by the author and Baker in \cite{bj} where it was shown that there exists a Bernoulli measure $\mu_{\q}$ such that
$$\dim \mu_{\q}= \sup_{\p \in \mathcal{P}} \dim \mu_{\p}.$$
Notice that by (\ref{less1}) this immediately implies the existence of a dimension gap, however it gives no quantitative information about the size of the gap.

In this paper we propose a new proof of the dimension gap. All objects which have been discussed so far have some interpretation in the language of thermodynamic formalism; for instance $\mu_{T}$ and $\mu_{\p}$ are \emph{Gibbs measures}, the dimension can typically be written in terms of the \emph{entropy}, and the \emph{variational principle} (\ref{vp}) describes the existence and uniqueness of a measure of maximal dimension. Therefore, it is a natural question to ask what is the meaning of a \emph{dimension gap} within the framework of thermodynamic formalism. As a consequence of the new proof that is given in this paper we demonstrate that a dimension gap corresponds to the existence of uniform lower bounds for the \emph{asymptotic variance} of a class of potentials. This is of particular interest since this appears to be a rare example of an application of lower bounds for the variance. We remark that while our approach does give some information about the size of the gap, since it does not improve on \cite{kpw} we will not make it explicit in order to keep our arguments concise. 

Throughout the paper we will assume that if $h(\mu_{\p})=-\sum_{n \in \N} p_n \log p_n < \infty$ then the entries $(p_n)_{n \in \N}$ of the probability vector $\p$ are decreasing and satisfy $p_n =O(\frac{1}{n^2})$ (meaning that there exists a constant $K>0$ such that $p_n \leq \frac{K}{n^2}$ for all $n$). To see that we can make the first assumption, suppose that for some $k \in \N$, $p_{k+1} < p_k$. Define $\p^{\prime}$ to be the probability vector given by
\begin{equation*}
\begin{array}{ccc}
p_n^{\prime}=p_n & \textnormal{if} & n \notin \{k, k+1\} \\
p_n^{\prime}= \frac{p_k+p_{k+1}}{2} &\textnormal{if} & n \in \{k, k+1\}.
\end{array}
\end{equation*}
Then since $h(\mu_{\p^{\prime}})> h(\mu_{\p})$ and $\chi(\mu_{\p^{\prime}})< \chi(\mu_{\p})$ (see for instance \cite[Lemma 3.5]{bj}), it follows that $\dim \mu_{\p^{\prime}}> \dim \mu_{\p}$. We can make the second assumption since given any probability vector $\p$ and any $\epsilon>0$, we can choose some probability vector $\q$ with the property that $q_n=0$ for all $n$ sufficiently large whose dimension `approximates' the dimension of $\mu_{\p}$, that is, $|\dim \mu_{\q}-\dim \mu_{\p}|< \epsilon$ (see for instance \cite[Proposition 3.6]{bj}). Since $\mu_{\q}$ is finitely supported, trivially $q_n=O(\frac{1}{n^2})$. Therefore it is sufficient to consider probability vectors that satisfy both assumptions on their weights.

Throughout this paper we denote 
\begin{equation}
\psi= \frac{1}{|T^{\prime}(z_1)|^{\frac{1}{4}}} \in (0,1) \label{psi}.
\end{equation}
Morally there are similarities with \cite{kpw} in the way in which the new proposed proof will be organised. To be precise, while Kifer, Peres and Weiss showed that it was enough to consider the dimension of the set of points in $\Gamma_{1}^{\delta_0}$ and $\Gamma_{11}^{\delta_0}$ which did not belong to $\mathcal{E}_{\lambda_0}$, we'll show that it is actually sufficient to study the dimension of Bernoulli measures whose probability vectors satisfy the following hypothesis for some constant $0<\epsilon< \psi <1$.

\begin{hyp}
The probability vector $\p$ satisfies $-\sum p_n \log p_n < \infty$ and additionally either
\begin{itemize}
\item[(a)] $p_1, p_2 > \epsilon$ or
\item[(b)] $p_1> \psi$.
\end{itemize}
\label{hyp1}
\end{hyp}

In particular, we'll show that there exists $c>0$ and $0<\epsilon< \psi <1$ such that whenever $\p$ does not satisfy Hypothesis \ref{hyp1} for this choice of $\epsilon$ then $\dim \mu_{\p}< 1-c$. Essentially this is down to the fact that if Hypothesis \ref{hyp1} is not satisfied, $\mu_{\p}$ must assign a lot of mass to a small neighbourhood of 0 (since the entries $p_n$ are decreasing) which allows us to bound the dimension of the measure directly from the fact that the Lyapunov exponent is forced to be large.

Consequently this allows us to restrict our attention to $\p$ which satisfy Hypothesis \ref{hyp1}. Fixing such $\p$, by
using tools from thermodynamic formalism we will show that we can relate $\dim \mu_{\p}$ to the derivative of a particular function $\beta_{\p}(t)$ at $t=1$. By using the properties of $\beta_{\p}$ we will show that the problem reduces to obtaining a lower bound on $\beta_{\p}^{\prime\prime}(t)$ which holds uniformly for all $t$ belonging to a compact interval and all $\p$ which satisfy Hypothesis \ref{hyp1}. In turn, this reduces to studying lower bounds on the asymptotic variance of a particular class of potentials, which comprises the main body of work in this paper. 

The paper is organised as follows. In section 2 we provide some preliminaries, including the necessary tools from thermodynamic formalism and some useful properties of the Gauss map. In section 3 we will show that there exist some constants $c, \epsilon_0>0$ such that if $h(\mu_{\p})< \infty$ and $\p$ does not satisfy Hypothesis \ref{hyp1} for $\epsilon=\epsilon_0$ or if $h(\mu_{\p})= \infty$ then $\dim \mu_{\p}< 1-c$. In particular, this will allow us to assume that Hypothesis \ref{hyp1} holds for $\epsilon=\epsilon_0$ for the remainder of the paper. In section 4 we obtain a bound on the dimension of measures which satisfy Hypothesis \ref{hyp1} (for $\epsilon=\epsilon_0$). In section 5 we tie the last two sections together to provide a proof of Theorem \ref{main}. Finally in section 6 we discuss a generalisation of Theorem \ref{main}.

\section{Preliminaries}

\subsection{Symbolic coding.} Let $\Sigma$, $\Sigma^{\ast}$, $\sigma$, $\Pi$ be defined as before. For $\i \in \Sigma^{\ast}$ let $|\i|$ denote the length of the word $\i$. For $\i, \j \in \Sigma$ let $\i \wedge \j \in \Sigma \cup \Sigma^{\ast}$ denote the longest initial block common to both $\i$ and $\j$. We equip $\Sigma$ with the metric $d$ given by $d(\i, \j)=\exp(-|\i \wedge \j|)$ if $|\i \wedge \j|< \infty$ and $d(\i, \j)=0$ otherwise. Given $\i=(i_n)_{n \in \N} \in \Sigma$, we let $\i|_n=i_1 \ldots i_n$ denote the finite word obtained by truncating $\i$ after $n$ digits. Given $i_1 \ldots i_n \in \Sigma^{\ast}$ let $(i_1 \ldots i_n)^{\infty}$ denote the unique periodic point $\i \in \Sigma$ of period $n$ for which $\i|_n=i_1 \ldots i_n$. Given a finite word $i_1 \ldots i_n$, denote $\I_{i_1\ldots i_n}= \Pi([i_1 \ldots i_n])$. Note that since $\I_n= [\frac{1}{n+1}, \frac{1}{n})$, $|\I_n|= \frac{1}{n(n+1)}=O(\frac{1}{n^2})$.

\subsection{Function spaces on $[0,1]$.} Let $C([0,1])$ denote all continuous functions $f:[0,1] \to \R$. Let 
$$[f]_1= \sup_{x \neq y} \frac{|f(x)-f(y)|}{|x-y|}$$
denote the Lipschitz constant of a function $f:[0,1] \to \R$. We say that $f$ is Lipschitz (continuous) if $[f]_1< \infty$. Let $\lip$ denote the space of all bounded Lipschitz continuous functions and equip this with the norm $\norm{\cdot}_{0,1}=[\cdot]_1+\norm{\cdot}_{\infty}$. 

We say that a potential $f: [0,1] \to \mathbb{R}$ is \emph{locally H\"older} if there exist constants $C>0$ and $0<\alpha<1$ such that for all $n \geq 1$ the variations $\textnormal{var}_n(f)$ decay exponentially:
\begin{eqnarray}
\textnormal{var}_n(f)= \sup_{i_1\ldots i_n \in \mathbb{N}^n} \left\{ |f(x)-f(y)|: x,y \in \I_{i_1,\ldots,i_n} \right\} \leq C\alpha^n.
\label{variations}
\end{eqnarray}
Note that $f$ being locally H\"older does not necessarily imply that it is bounded. We define
$$\mathcal{H}_{\alpha}=\left\{f:[0,1] \to \R : \textnormal{$f$ is bounded and  $\sup_n \frac{\textnormal{var}_n(f)}{\alpha^n}< \infty$}\right\} $$
and denote the space of all bounded locally H\"older functions by $\H= \cup_{0<\alpha<1} \H_{\alpha}$.  If $f \in \mathcal{H}_{\alpha}$, define the seminorm $[f]_{\alpha}$ to be the smallest constant $C$ that one can take in (\ref{variations}) and we equip $\mathcal{H}_{\alpha}$ with the norm $\norm{\cdot}_{\alpha}= [\cdot]_{\alpha}+ \norm{\cdot}_{\infty}$.

We say that a locally H\"older potential $f: [0,1] \to \R$ is \emph{summable} if 
\begin{eqnarray}
\sum_{n \in \mathbb{N}} \exp(\sup f|_{\I_{n}}) < \infty.
\label{summable}
\end{eqnarray}

\subsection{Thermodynamic formalism.} We can define the \emph{topological pressure} of a potential $g$ as follows.

\begin{defn}[Topological pressure]
Let $g: [0,1] \to \mathbb{R}$ be a locally H\"older potential. Then the pressure of $g$ is given by
$$P(g) = \lim_{n \to \infty} \frac{1}{n} \log \left( \sum_{x: T^nx=x } \exp(S_ng(x))\right)$$
where $S_ng(x)$ denotes the Birkhoff sum $S_ng(x)=g(x)+ \ldots g(T^{n-1}x)$.
\end{defn}

In general, the pressure of $g$ can either be finite or infinite, but if $g$ is summable then $P(g)< \infty$.

Given a locally H\"older potential $g:[0,1] \to \R$, we say that a measure $\mu_g$ is a \emph{Gibbs measure} for $g$ if there exist constants $C, P> 0$ such that for all $n \in \N$, $\i \in \Sigma^{\ast}$ and $x \in \I_{\i}$,
\begin{eqnarray}
C^{-1} \leq \frac{\mu_g(\I_{\i})}{\exp(S_ng(x)-nP)} \leq C. \label{gibbs}
\end{eqnarray}
Note that we do not require $\mu_g$ to be invariant.

By \cite[Corollary 2.10]{murpf} we know about the existence of $T$-invariant Gibbs measures.

\begin{prop}[Existence of Gibbs measures]
Let $g:[0,1] \to \R$ be a locally H\"older summable potential. Then there exists a unique $T$-invariant (probability) Gibbs measure $\mu_g$ for $g$. 
Moreover, the constant $P$ in (\ref{gibbs}) is given by $P=P(g)$. \label{exist}
\end{prop}

Gibbs measures have a useful characterisation via the Ruelle-Perron-Frobenius theorem, see \cite[Corollary 2.10]{murpf}.

\begin{prop}[Ruelle-Perron-Frobenius theorem] \label{rpf}
Let $g:[0,1] \to \R$ be a locally H\"older potential with $P(g)=0$ and let $\l_g: \mathcal{H} \to \mathcal{H}$ be the transfer operator given by
$$\l_g f(x)=\sum_{Ty=x} \exp(g(y))f(y).$$
Then there exists a unique (positive) eigenfunction $\l_gh=h$ and a unique eigenmeasure $\l_g^{\ast} \tilde{\mu}= \tilde{\mu}$, where $\l^{\ast}_g$ denotes the dual of $\l_g$. Moreover $\tilde{\mu}$ is a Gibbs measure for $g$. Let $\m_g: \mathcal{H} \to \mathcal{H}$ be the normalised operator defined by
$$\m_g f(x)=\frac{1}{h(x)}\sum_{Ty=x} h(y) \exp(g(y))f(y)$$
so that $\m_g \one =\one$. Then $\textup{d}\mu= h \textup{d} \tilde{\mu}$ is the unique $T$-invariant Gibbs measure for $g$ and $\m_g^{\ast} \mu=\mu$.
\end{prop}

Given $u \in \mathcal{H}$ we call $u -u \circ T$ a \emph{coboundary}. We say that two locally H\"older functions $f, g: [0,1] \to \R$ are \emph{cohomologous} (denoted by $f \sim g$) if there exists some function $u \in \H$ such that
$$f=g+u-u\circ T.$$

\subsection{Regularity of $T$.} It is easy to check that for all $x \in [0,1]$, $|(T^2)^{\prime}(x)| \geq \frac{9}{4}$. That means that although $T$ is itself not uniformly expanding, the second iterate $T^2$ is. Since $T^{\prime}(x)= -x^{-2}$ and $T^{\prime\prime}(x)=2x^{-3}$ it follows easily that
\begin{eqnarray}
\sup_{n \in \N} \sup_{x, y, z \in \I_n} \left|\frac{T^{\prime\prime}(x)}{T^{\prime}(y)T^{\prime}(z)}\right| = 16.
\label{renyi1}
\end{eqnarray}
Consequently, one can use (\ref{renyi1}) to show that $-\log|T^{\prime}|$ is locally H\"older; in particular $-\log|T^{\prime}| \in \H_{\frac{2}{3}}$. Throughout the rest of the paper we fix $\alpha= \frac{2}{3}$. A consequence of the H\"older regularity of $-\log|T^{\prime}|$ is the following useful \emph{bounded distortion} property, see for instance \cite[\S 7.4 Lemma 2]{cfs}.

\begin{prop}[Bounded distortion property] \label{bd}
There exists some $C>0$ such that for all $n \in \N$, $i_1 \ldots i_n \in \Sigma^{\ast}$ and $x, y \in \I_{i_1 \ldots i_n}$,
$$C^{-1} \leq \frac{(T^n)^{\prime}(x)}{(T^n)^{\prime}(y)} \leq C.$$
In particular for any $x \in \I_{i_1 \ldots i_n}$,
$$C^{-1} \leq \frac{|(T^n)^{\prime}(x)|}{|\I_{i_1 \ldots i_n}|^{-1}} \leq C.$$
\end{prop}

\section{Measures that do not satisfy Hypothesis \ref{hyp1}} \label{tail}

In this section we show that there exists some $c, \epsilon>0$ such that if $\p$ does not satisfy Hypothesis \ref{hyp1} for this choice of $\epsilon$, then $\dim \mu_{\p}< 1-c$.

Given $\lambda>0$ recall that $\mathcal{E}_{\lambda}$ was defined to be
\begin{eqnarray} \label{e}
\mathcal{E}_{\lambda}= \bigcap_{j=1}^{\infty} \bigcup_{n=j}^{\infty} \left\{x \in (0,1): |J_n(x)| \leq \exp(-\lambda n)\right\}.
\end{eqnarray}

For $\frac{1}{2}<s<1$ we denote
\begin{eqnarray}
\kappa(s)= \log \left( \sup_{x \in (0,1)} \sum_{n \in \mathbb{N}} \frac{1}{|T^{\prime}(T_n^{-1}x)|^s} \right) < \infty.
\label{q}
\end{eqnarray}

By \cite[Theorem 4.1]{kpw}, for any $\lambda>0$ and $\frac{1}{2}<s<1$
\begin{eqnarray}
\dim \mathcal{E}_{\lambda} \leq s+\frac{\kappa(s)}{\lambda}.
\label{kpw bound}
\end{eqnarray}

We begin by using (\ref{kpw bound}) to show that any measure with infinite entropy (and therefore infinite Lyapunov exponent) will have dimension at most $\frac{1}{2}$.

\begin{lma}
Let $\mu_{\p}$ be a Bernoulli measure such that $h(\mu_{\p})= \infty$. Then
\[\dim \mu_{\p} \leq \frac{1}{2}.\]
\label{infinite thm}
\end{lma}

\begin{proof}
Let $h(\mu_{\p})= \infty$. Then $\chi(\mu_{\p})=\infty$. Thus for $\mu_{\p}$ almost every $x$,
$$\liminf_{n \to \infty} \frac{1}{n} \sum_{k=0}^{n-1} \log |T^{\prime}(T^k(x))| = \infty.$$
Fix $\lambda>0$. Then for $\mu_{\p}$ almost every $x$
\begin{eqnarray}
\frac{1}{n} \sum_{k=0}^{n-1} \log (T^{\prime}(T^k(x))) > 2\lambda
\label{lambda}
\end{eqnarray}
for all $n$ sufficiently large. By rearranging (\ref{lambda}) we obtain that for all $x$ that satisfy (\ref{lambda}), there exists a subsequence $n_k$ such that
$$|(T^{n_k})^{\prime}(x)|^{-1} < \exp(-2\lambda n_k)$$
for all $k \in \N$. By Proposition \ref{bd} this implies that
$$|J_{n_k}(x)| \leq C|(T^{n_k})^{\prime}(x)|^{-1} \leq C\exp(-2\lambda n_k) \leq \exp(-\lambda n_k)$$
along the subsequence $n_k$, provided $\lambda$ is sufficiently large. Therefore $x \in \mathcal{E}_{\lambda}$ which implies that $\mu_{\p}(\mathcal{E}_{\lambda})=1$ since we were considering $x$ that belong to a set of full measure.

Let $\frac{1}{2}< s<1$. By (\ref{kpw bound}), $\dim \mathcal{E}_{\lambda} \leq s+ \frac{\kappa(s)}{\lambda}$. Since $\mu_{\p}(\mathcal{E}_{\lambda})=1$ for all $\lambda$, it follows that $\dim \mu_{\p} \leq s+ \frac{\kappa(s)}{\lambda}$ where $\kappa(s)$ is given by (\ref{q}). Since $s$ can be chosen arbitrarily close to $\frac{1}{2}$ and $\lambda$ can be chosen to be arbitrarily large, the result follows.
\end{proof}

We can use similar ideas to consider measures with finite entropy whose associated probability vectors do not satisfy Hypothesis \ref{hyp1}.

\begin{lma}
Fix $\frac{1}{2}<s_0<1$. Let $\lambda_0>0$ such that $s_0+\frac{\kappa(s_0)}{\lambda_0}<1$. Then there exists $\epsilon_0>0$ such that if $h(\mu_{\p})< \infty$ and $\p$ does not satisfy Hypothesis \ref{hyp1} for $\epsilon=\epsilon_0$ then
$$\dim \mu_{\p} \leq s_0+\frac{\kappa(s_0)}{\lambda_0}.$$
\label{tail lemma}
\end{lma}

\begin{proof}
Fix $\lambda_0$ sufficiently large that $s_0+\frac{\kappa(s_0)}{\lambda_0}<1$. Fix $N$ sufficiently large that 
$$\frac{1-\psi}{2} \inf_{x \in \I_N} \log |T^{\prime}(x)|> 2\lambda_0,$$ where $\psi$ was defined in (\ref{psi}). Fix $\epsilon_0$ sufficiently small that $\epsilon_0< \frac{1-\psi}{2N}$. Since the $p_n$ are decreasing it follows that $\sum_{n=N+2}^{\infty} p_n \geq 1-\psi-N \epsilon_0$. Thus since $\epsilon_0< \frac{1-\psi}{2N}$,
$$\int \log|T^{\prime}|d\mu_{\p} \geq \frac{1-\psi}{2} \inf_{x \in \I_N} \log |T^{\prime}(x)|> 2\lambda_0.$$
As in the proof of Lemma \ref{infinite thm} this implies that $\mu_{\p}$ almost every $x$ belongs to $\mathcal{E}_{\lambda_0}$ and therefore $\dim \mu_{\p} \leq s_0+\frac{\kappa(s_0)}{\lambda_0}$. 
\end{proof}

\section{Measures that satisfy Hypothesis \ref{hyp1}} \label{top}

Throughout this section we fix $\epsilon=\epsilon_0$ given by Lemma \ref{tail lemma} and we fix a probability vector $\p$ that satisfies the following hypothesis.

\begin{hyp} \label{hyp} The probability vector $\p$ satisfies that $\dim \mu_{\p}> \frac{3}{4}$ and additionally either
\begin{itemize}
 \item[(a)]  $p_1, p_2 \geq \epsilon$ or
 \item[(b)] $p_1> \psi$.
 \end{itemize}
\end{hyp}

If $\p$ satisfies Hypothesis \ref{hyp} we may also say that $\mu_{\p}$ satisfies Hypothesis \ref{hyp}. Note that by Lemma \ref{infinite thm}, $\dim \mu_{\p}> \frac{3}{4}$ implies that $h(\mu_{\p})< \infty$ and so Hypothesis \ref{hyp} is slightly stronger than Hypothesis \ref{hyp1} (in particular, if $\p$ satisfies Hypothesis \ref{hyp1} then either $\dim \mu_{\p} \leq \frac{3}{4}$ or $\p$ satisfies Hypothesis \ref{hyp}). Also, since $h(\mu_{\p})< \infty$ we have $\dim \mu_{\p}= \frac{h(\mu_{\p})}{\chi(\mu_{\p})}$. To make our arguments clearer we also assume that $p_n >0$ for all $n$, although the proof could be easily adapted without this extra assumption.

The main result in this section is that we can obtain a uniform upper bound on the dimension of any measure $\mu_{\p}$ whose probability vector satisfies Hypothesis \ref{hyp}.

\begin{lma}
There exists $\eta_1>0$ such that for any $\mu_{\p}$ that satisfies Hypothesis \ref{hyp},
$$\dim \mu_{\p} \leq 1-\eta_1.$$
\label{top bound}
\end{lma}

The method used in this section is based on an approach which was proposed by Kesseb\"ohmer, Stratmann and Urba\'nksi and was outlined in a talk given by Kesseb\"ohmer in \cite{kess}.

For a fixed probability vector $\p$ define the Bernoulli potential $f_{\p} : [0,1] \setminus \Q \to (-\infty, 0]$ by
$$f_{\p}= \sum_{n \in \N} \log p_n \one_{\I_n}.$$

Notice that $f_{\p}$ is the Gibbs potential for the Bernoulli measure $\mu_{\p}$. We are now ready to introduce the function $\beta_{\p}$. 

\begin{defn}
Fix a probability vector $\p$ that satisfies Hypothesis \ref{hyp}. We can define the function $\beta_{\p}: [0,1] \to [0,1]$ where $\beta_{\p}(t)$ is defined implicitly as the solution to
\begin{eqnarray} P(-\beta_{\p}(t)\log |T^{\prime}|+tf_{\p})=0.
\label{bp}
\end{eqnarray}
\end{defn}

Note that it is not immediately obvious that $\beta_{\p}$ should be well-defined; this fact will follow from Proposition \ref{bp properties}. 

\begin{figure}
	\centering
	\includegraphics[width=80mm]{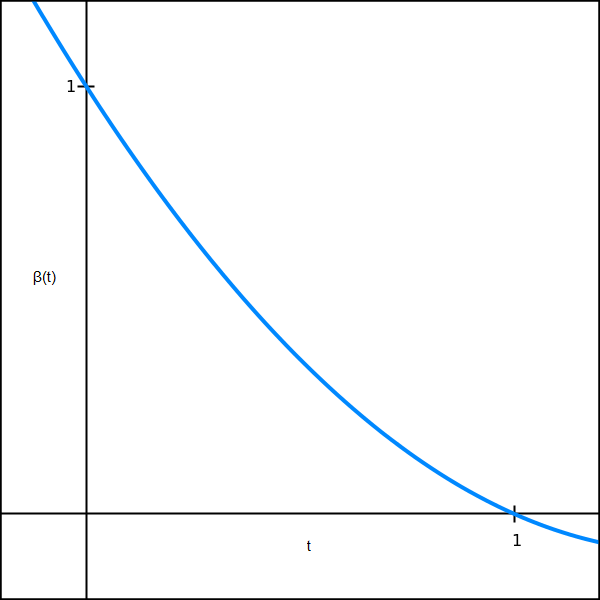}
	\caption{A typical graph of $\beta_{\p}(t)$.}
	\label{beta}
\end{figure}

We denote the function that appears inside the pressure in (\ref{bp}) by $g_{\p,t}: [0,1] \setminus \Q \to \R$
\begin{eqnarray}
g_{\p,t}= -\beta_{\p}(t)\log |T^{\prime}| + tf_{\p} .
\label{gpt}
\end{eqnarray}
By Proposition \ref{gibbs} we know that there exists a unique invariant Gibbs measure for $g_{\p,t}$ which we will denote by $\mu_{\p,t}$. 

The function $\beta_{\p}$ will be the object of our focus throughout this section. In the following proposition we summarise its important properties.

\begin{prop} \label{bp properties}
The function $\beta_{\p}: [0,1] \to [0,1]$ satisfies the following properties:
\begin{enumerate}
\item $\beta_{\p}(t)$ is convex and decreasing on $[0,1]$.
\item $\beta_{\p}(0) = 1$ and $\beta_{\p}(1)=0$.\item $\beta_{\p}(t)$ is analytic for $t \in [0,1]$. Moreover the first derivative of $\beta_{\p}$ (with respect to $t$) is given by
\begin{eqnarray}
\beta_{\p}^{\prime}(t)= \frac{-\int f_{\p}d\mu_{\p,t}}{\int \log |T^{\prime}|d\mu_{\p,t}}
\label{beta1}
\end{eqnarray}
(so in particular $\dim \mu_{\p}= |\beta_{\p}^{\prime}(1)|$)
and the second derivative is given by
\begin{eqnarray}
\beta_{\p}^{\prime\prime}(t) = \frac{\sigma^2_{\mu_{\p,t}}(-\beta_{\p}^{\prime}(t)\log |T^{\prime}|+f_{\p})}{\int \log|T^{\prime}|d\mu_{\p,t}}
\label{beta2}
\end{eqnarray}
where the variance $\sigma_{\mu_{\p,t}}^2(-\beta_{\p}^{\prime}(t)\log |T^{\prime}|+f_{\p})$ is given by
\begin{eqnarray} \label{variance}
\sigma^2_{\mu_{\p,t}}(f_{\p,t})= \int f_{\p,t}^2\textup{d}\mu_{\p,t}+ 2\sum_{n=1}^{\infty}\left(\int f_{\p,t} \cdot f_{\p,t} \circ T^n \textup{d}\mu_{\p,t}\right) .
\end{eqnarray}
\end{enumerate}
Moreover, these properties determine the graph of $\beta_{\p}(t)$; see Figure \ref{beta}.
\end{prop}

\begin{proof}
It is easy to see that $\beta_{\p}$ is decreasing, since
$$P(-\beta_{\p}(t)\log|T^{\prime}|+tf_{\p})= \lim_{n \to \infty} \frac{1}{n} \log \left( \sum_{i_1,\ldots i_n \in \N^n}\frac{(p_{i_1} \dots p_{i_n})^t}{|(T^n)^{\prime}(\Pi((i_1 \ldots i_n)^{\infty}))|^{\beta_{\p}(t)}}\right).$$
To see that $\beta_{\p}$ is convex, notice that for any $n \in \mathbb{N}$, and $a, u, t \in (0,1)$
\begin{eqnarray*}
\sum_{i_1,\ldots, i_n \in \mathbb{N}^n} \frac{p_{i_1}^{at}\ldots p_{i_n}^{at}}{|(T^n)^{\prime}(\Pi((i_1\ldots i_n)^{\infty}))|^{a\beta_{\p}(t)}}\frac{p_{i_1}^{(1-a)u}\ldots p_{i_n}^{(1-a)u}}{|(T^n)^{\prime}(\Pi((i_1\ldots i_n)^{\infty}))|^{(1-a)\beta_{\p}(u)}} \leq \\
\left(\sum_{i_1,\ldots, i_n \in \mathbb{N}^n} \frac{p_{i_1}^{t}\ldots p_{i_n}^{t}}{|(T^n)^{\prime}(\Pi((i_1\ldots i_n)^{\infty}))|^{\beta_{\p}(t)}}\right)^a\left(\sum_{i_1,\ldots, i_n \in \mathbb{N}^n} \frac{p_{i_1}^{u}\ldots p_{i_n}^{u}}{|(T^n)^{\prime}(\Pi((i_1\ldots i_n)^{\infty}))|^{\beta_{\p}(u)}}\right)^{1-a}
\end{eqnarray*}
by H\"older's inequality. Therefore
\begin{align*}
P(-(a\beta_{\p}(t)&+(1-a)\beta_{\p}(u))\log|T^{\prime}|+(at+(1-a)u)f_{\p}) \\
&\leq aP(-\beta_{\p}(t)\log |T^{\prime}|+tf_{\p}) +(1-a)P(-\beta_{\p}(u)\log |T^{\prime}|+uf_{\p})=0.
\end{align*}
Therefore it follows that $\beta_{\p}(at+(1-a)u) \leq a\beta_{\p}(t)+(1-a)\beta_{\p}(u)$ since when $t$ is fixed, $P(-b\log |T^{\prime}|+tf_{\p})$ is decreasing in $b$. 

For the second part, by using Proposition \ref{bd} it is easy to see that $P(-\log |T^{\prime}|)=0$ which implies that $\beta_{\p}(0)=1$. Similarly it is easy to see that $P(f_{\p})=0$ thus it follows that $\beta_{\p}(1)=0$. 

To prove the third part, we begin by showing that $\beta_{\p}$ is analytic in a neighbourhood of 1. Let $r< \frac{1}{2}$. By Proposition \ref{bd} and the fact that $p_n =O( \frac{1}{n^2})$, $$\sum_{n \in \N} p_n^{1-r} =O \left(\sum_{n \in \N} \frac{1}{n^{2(1-r)}}\right)$$ 
and therefore is a finite sum.
Let $(t, b) \in [1-\frac{r}{2}, 1+\frac{r}{2}] \times [-\frac{r}{2}, \frac{r}{2}]$. Then since $p_n=O(\frac{1}{n^2})$ there exists a constant $K>0$ such that 
\begin{align*}
P(-b\log &|T^{\prime}|+tf_{\p}) \\
 &\leq \lim_{n \to \infty} \frac{1}{n} \log \left( \sum_{i_1\ldots i_n \in \N^n} (p_{i_1}\ldots p_{i_n})^{1-\frac{r}{2}} |(T^n)^{\prime}(\Pi(i_1\ldots i_n)^{\infty})|^{\frac{r}{2}} \right) \\
&\leq \lim_{n \to \infty} \frac{1}{n} \log \left( \sum_{i_1\ldots i_n \in \N^n} (p_{i_1}\ldots p_{i_n})^{1-r} (p_{i_1}\ldots p_{i_n})^{\frac{r}{2}}(|\I_{i_1}|\ldots |\I_{i_n}|)^{-\frac{r}{2}} C^{n\frac{r}{2}}  \right) \\
&\leq \lim_{n \to \infty} \frac{1}{n} \log \left( \sum_{i_1\ldots i_n \in \N^n} (p_{i_1}\ldots p_{i_n})^{1-r}  K^{n}C^{n\frac{r}{2}} \right) \\
&= \frac{r}{2}\log C + \log K + \log\left(\sum_{n \in \N} p_n^{1-r}\right) < \infty
\end{align*}
where the second inequality follows by Proposition \ref{bd}. Therefore, by \cite[Theorem 2.6.12]{mu} $P(-b\log |T^{\prime}|+tf_{\p})$ is analytic for all $(t, b) \in [1-\frac{r}{2}, 1+\frac{r}{2}] \times [-\frac{r}{2}, \frac{r}{2}]$ and by the implicit function theorem $\beta_{\p}(t)$ is analytic for all $t \in (1-\frac{r}{2}, 1+\frac{r}{2})$. We will return to show that $\beta_{\p}$ is analytic on the whole interval $[0,1]$ after verifying that (\ref{beta1}) holds for any $t \in (1-\frac{r}{2}, 1+\frac{r}{2})$ (and indeed for all $t$ at which $\beta_{\p}(t)$ is analytic).

To verify (\ref{beta1}) we follow the arguments of Ruelle \cite{ruelle}. Fix $t$ such that $\beta_{\p}$ is analytic at $t$. We differentiate (\ref{bp}) and apply \cite[Proposition 2.6.13]{mu} and the implicit function theorem to deduce that
\begin{eqnarray}
-\beta_{\p}^{\prime}(t) \int \log |T^{\prime}|d\mu_{\p,t} + \int f_{\p} d\mu_{\p,t} =0.
\label{deriv}
\end{eqnarray}
In particular, since $\beta_{\p}(1)=0$, it follows that $\mu_{\p,1}=\mu_{\p}$ and therefore
$$\dim \mu_{\p}= \frac{h(\mu_{\p})}{\chi(\mu_{\p})}=-\frac{\int f_{\p}d\mu_{\p,1}}{\int \log|T^{\prime}|d\mu_{\p,1}}= -\beta_{\p}^{\prime}(1)=|\beta_{\p}^{\prime}(1)|.$$

Using this we can now show that in fact $\beta_{\p}(t)$ is analytic for \emph{all} $t \in [0,1]$. By Hypothesis \ref{hyp} $|\beta_{\p}^{\prime}(1)|=\dim \mu_{\p}> \frac{1}{2}$, therefore it follows by convexity of $\beta_{\p}$ that $\beta_{\p}(t)> \frac{1}{2} (1-t)$ for all $t \in [0,1]$. In particular for all $t \in [0,1]$
\begin{eqnarray}
\beta_{\p}(t)+t \geq \beta_{\p}(t)+\frac{1}{2} t> \frac{1}{2}.
\label{beta lb}
\end{eqnarray}
Fix $t$ and choose $\epsilon$ sufficiently small so that $\beta_{\p}(t)+t-2\epsilon> \frac{1}{2}$. Then for all $(u, b) \in (t-\epsilon, t+\epsilon) \times (\beta_{\p}(t)-\epsilon, \beta_{\p}(t)+\epsilon)$,
\begin{eqnarray*}
P(-b\log|T^{\prime}|+uf_{\p})&=& \lim_{n \to \infty} \frac{1}{n} \log \left(\sum_{\i \in \N^n} \frac{p_{\i}^u}{|(T^n)^{\prime}(\Pi((\i)^{\infty}))|^{b}}\right) \\
&\leq& C'+ \log \left(\sum_{n \in \N} \frac{1}{n^{2(b+u)}}\right)< \infty 
\end{eqnarray*}
where $C'$ is a constant coming from Proposition \ref{bd} and the fact that $p_n=O(\frac{1}{n^2})$, and the final inequality is because $b+u \geq \beta_{\p}(t)+t-2\epsilon>\frac{1}{2}$.

By the implicit function theorem and \cite[Theorem 2.6.12]{mu}, $\beta_{\p}(t)$ is analytic for all $t \in [0,1]$, and the derivative $\beta_{\p}^{\prime}(t)$ satisfies (\ref{beta1}) by the same arguments as before.

To verify (\ref{beta2}) we differentiate (\ref{deriv}) to obtain
$$\beta_{\p}^{\prime\prime}(t) \int \log|T^{\prime}| \textup{d}\mu_{\p,t} + \beta_{\p}^{\prime}(t) \frac{\textup{d} \left( \int \log |T^{\prime}| \textup{d}\mu_{\p,t}\right)}{\textup{d} t} - \frac{\textup{d}\left(\int f_{\p} \textup{d}\mu_{\p,t}\right)}{\textup{d} t}=0.$$
By \cite[Proposition 2.6.14]{mu}, 
$$\frac{\textup{d} \left( \int \log |T^{\prime}| \textup{d}\mu_{\p,t}\right)}{\textup{d} t}= \sigma^2_{\mu_{\p,t}}( -\beta_{\p}^{\prime}(t)\log|T^{\prime}|+f_{\p}, \log|T^{\prime}|) $$
and
$$\frac{\textup{d}\left(\int f_{\p} \textup{d}\mu_{\p,t}\right)}{\textup{d} t}= \sigma^2_{\mu_{\p,t}}( -\beta_{\p}^{\prime}(t)\log|T^{\prime}|+f_{\p}, f_{\p}) $$
and therefore
\begin{eqnarray}
\beta_{\p}^{\prime\prime}(t) = \frac{\sigma^2_{\mu_{\p,t}}(-\beta_{\p}^{\prime}(t)\log |T^{\prime}|+f_{\p})}{\int \log|T^{\prime}|\textup{d}\mu_{\p,t}} \geq 0.
\label{deriv2}
\end{eqnarray}
By (\ref{deriv}), $\mu_{\p,t}(-\beta_{\p}^{\prime}(t)\log|T^{\prime}|+f_{\p})=0$ for all $\p$ and $t$ (where $\mu_{\p,t}(f)$ denotes $\int f \textup{d} \mu_{\p,t}$), thus $\sigma_{\mu_{\p,t}}^2(f_{\p,t})$ is given by (\ref{variance}).

\end{proof}

By rewriting $\dim \mu_{\p}$ as the absolute value of the derivative of $\beta_{\p}$ at 1, we are now able to exploit the tools of calculus to find an upper bound on $|\beta_{\p}^{\prime}(1)|=\dim \mu_{\p}$. In particular we are interested in showing that $\beta_{\p}$ is `uniformly convex' in some compact interval of $t$. Therefore we need to obtain lower bounds on $\beta_{\p}^{\prime\prime}(t)$ which are uniform over all $\p$ which satisfy Hypothesis \ref{hyp} and all $t$ belonging to some compact interval.

From now on we shall denote $f_{\p,t}: [0,1] \setminus \Q \to \R$ by
\begin{eqnarray}
f_{\p,t}= -\beta_{\p}^{\prime}(t)\log |T^{\prime}|+f_{\p}.
\label{fpt}
\end{eqnarray}

By (\ref{deriv2}), we are interested in finding an upper bound for the Lyapunov exponent $\chi(\mu_{\p,t})$ and a lower bound for the variance $\sigma^2_{\mu_{\p,t}}(f_{\p,t})$ which henceforth we will denote by $\sigma_{\p,t}^2(f_{\p,t})$. The Lyapunov exponent is not difficult to estimate from above, but we will delay this until Lemma \ref{e2 proof}. Instead, our primary focus will be obtaining a lower bound for the variance. It is well known that the variance satisfies
\begin{eqnarray} \label{gk fpt}
\sigma^2_{\p,t}(f_{\p,t})= \int \tilde{f}_{\p,t}^2\textup{d}\mu_{\p,t}+ 2\sum_{n=1}^{\infty}\left(\int \tilde{f}_{\p,t} \cdot \tilde{f}_{\p,t} \circ T^n \textup{d}\mu_{\p,t}\right) 
\end{eqnarray}
for any function $\tilde{f}_{\p,t}$ which is cohomologous to $f_{\p,t}$. The second term on the right hand side of (\ref{gk fpt}) is what makes it difficult to study lower bounds on the variance. Therefore, our aim is to find a coboundary $U_{\p,t}-U_{\p,t} \circ T$ such that if we substitute $\tilde{f}_{\p,t}=f_{\p,t}+U_{\p,t}-U_{\p,t}\circ T$ into (\ref{gk fpt}) then the right hand term will vanish. Therefore, in the first part of this section we introduce a family of transfer operators which will aid us towards obtaining the appropriate function $U_{\p,t}$ for which $\sigma_{\p,t}^2(f_{\p,t})= \int \tilde{f}_{\p,t}^2\textup{d}\mu_{\p,t}$. To this end, we introduce a family of transfer operators.

\begin{defn}
For a fixed $\p$ and $t$ we  define the bounded linear operator $\l_{\p,t}: \mathcal{H} \to \mathcal{H}$ by
$$\l_{\p, t} w(x)= \sum_{Ty=x} \exp(g_{\p,t}(y))w(y).$$
Note that this can be written alternatively as
$$\l_{\p, t} w(x)= \sum_{n \in \mathbb{N}} \exp(g_{\p,t}(T_n^{-1}x))w(T_n^{-1}x).$$ 
\end{defn}

Notice that each operator in the family above is well-defined since 
$$\sum_{n \in \N} \exp(g_{\p,t}(T_n^{-1}x))= \sum_{n \in \N} \frac{p_n^t}{|T^{\prime}(T_n^{-1}x)|^{\beta_{\p}(t)}}< \infty.$$ 

It will be more convenient for us to work with the normalised transfer operator. 

\begin{prop} \label{rpf2}
There exists a normalised operator $\m_{\p,t}: \mathcal{H} \to \mathcal{H}$ given by
$$\m_{\p,t} w= h_{\p,t}^{-1}\l_{\p,t}(h_{\p,t} w)$$
such that $\m_{\p,t} \one=\one$, where $h_{\p,t}$ is the unique fixed point of $\l_{\p,t}$. Equivalently $\m_{\p,t}w=\sum_{Ty=x} \exp(\tilde{g}_{\p,t}(y))w(y)$ where $\tilde{g}_{\p,t}=g_{\p,t}+h_{\p,t}-h_{\p,t} \circ T$. Moreover, $\m_{\p,t}^{\ast} \mu_{\p,t}= \mu_{\p,t}$ and $\textup{d}\mu_{\p,t}=h_{\p,t}\textup{d} \tilde{\mu}_{\p,t}$ where $\l_{\p,t}^{\ast}\tilde{\mu}_{\p,t}=\tilde{\mu}_{\p,t}$. 
\end{prop}

\begin{proof}
This is essentially a restatement of Proposition \ref{rpf}.
\end{proof}

When seeking \emph{upper} estimates on the variance, the second term on the right hand side in (\ref{gk fpt}) can easily be dealt with, for instance one can bound it above by knowing an explicit rate for the decay of the correlation functions. However, when one is interested in \emph{lower} estimates, this term makes the variance difficult to bound from below. Since (\ref{gk fpt}) holds for any $\tilde{f}_{\p,t}$ which is cohomologous to $f_{\p,t}$, it would be useful if we could find some $\tilde{f}_{\p,t} \sim f_{\p,t}$ for which
$$\int \tilde{f}_{\p,t} \cdot \tilde{f}_{\p,t} \circ T^n \textup{d}\mu_{\p,t}=0$$
for all $n \in \N$. Since $\m_{\p,t}^{\ast} \mu_{\p,t}=\mu_{\p,t}$ and $\m_{\p,t}^n(\tilde{f}_{\p,t} \cdot \tilde{f}_{\p,t} \circ T^n)= \tilde{f}_{\p,t} \m_{\p,t}^n(\tilde{f}_{\p,t})$ we can rewrite the above as
\begin{eqnarray*}
\int \tilde{f}_{\p,t} \cdot \tilde{f}_{\p,t} \circ T^n \textup{d}\mu_{\p,t}&=& \int \m_{\p,t}^n(\tilde{f}_{\p,t} \cdot \tilde{f}_{\p,t} \circ T^n)\textup{d}\mu_{\p,t} \\
&=& \int \tilde{f}_{\p,t} \cdot \m_{\p,t}^n(\tilde{f}_{\p,t}) \textup{d}\mu_{\p,t}=0.
\end{eqnarray*}
Writing $\tilde{f}_{\p,t}=f_{\p,t}  + U_{\p,t} - U_{\p,t} \circ T$ for some coboundary $ U_{\p,t} - U_{\p,t} \circ T$, it transpires that the property we want is $\m_{\p,t}(f_{\p,t}  + U_{\p,t} - U_{\p,t} \circ T)=0$. This leads us to the following definition for $U_{\p,t}$, which we now fix.

\begin{defn}
Define $U_{\p,t}:[0,1] \setminus \Q \to \R$ by
$$U_{\p,t}= \sum_{n=1}^{\infty}  \m_{\p,t}^n\left(f_{\p,t}\right)$$
and $\tilde{f}_{\p,t}:[0,1] \setminus \Q \to \R$ by
$$\tilde{f}_{\p,t}= f_{\p,t} + U_{\p,t}- U_{\p,t} \circ T.$$
\label{coboundary}
\end{defn}

It will be a consequence of Lemma \ref{lemma r} that $U_{\p,t} \in \H_{\alpha}$ (although it is already not difficult to see this: it is easy to show that $\m_{\p,t} f_{\p,t} \in \H_{\alpha}$, and therefore by \cite[Theorem 2.4.6]{mu} one can deduce that $\norm{\m_{\p,t}^n f_{\p,t}}_{\alpha}$ decays exponentially fast in $n$). As suggested above, it turns out that this definition for $U_{\p,t}$ fits our purposes.

\begin{lma} \label{coboundary good}
For all $\p$ and $t$,
$$\m_{\p,t}(\tilde{f}_{\p,t})=\m_{\p,t}(f_{\p,t}+U_{\p,t}-U_{\p,t} \circ T)=0.$$
\end{lma}

\begin{proof}
It follows from definition that
\begin{eqnarray*}
\m_{\p,t}(\tilde{f}_{\p,t}) &=& \m_{\p,t}(f_{\p,t})+ \m_{\p,t}(U_{\p,t}) - \m_{\p,t}(U_{\p,t} \circ T) \\
&=& \m_{\p,t}(f_{\p,t})+ \sum_{n=2}^{\infty} \m_{\p,t}^n(f_{\p,t}) - \sum_{n=2}^{\infty} \m_{\p,t}^n(f_{\p,t} \circ T) \\
&=& \sum_{n=1}^{\infty} \m_{\p,t}^n(f_{\p,t}) -\sum_{n=2}^{\infty} \m_{\p,t}^n(f_{\p,t} \circ T) \\
&=& \sum_{n=1}^{\infty} \m_{\p,t}^n(f_{\p,t}) -\sum_{n=2}^{\infty} \m_{\p,t}^{n-1}(\m_{\p,t}(f_{\p,t} \circ T))\\
&=& \sum_{n=1}^{\infty} \m_{\p,t}^n(f_{\p,t})- \sum_{n=2}^{\infty} \m_{\p,t}^{n-1}(f_{\p,t} \cdot \m_{\p,t}(\mathbf{1})) \\
&=& 0.
\end{eqnarray*}
\end{proof}

As an immediate corollary to the above, we can write the variance as a single integral as we intended.

\begin{cor}
We can write
$$\sigma_{\p,t}^2(f_{\p,t})=\int \tilde{f}_{\p,t}^2 \textup{d}\mu_{\p,t}.$$
\label{rewrite2}
\end{cor}

\begin{proof}
By (\ref{gk fpt})
$$\sigma_{\p,t}^2(f_{\p,t})= \int \tilde{f}_{\p,t}^2\textup{d}\mu_{\p,t} + 2 \sum_{n=1}^{\infty} \int \tilde{f}_{\p,t} \cdot \tilde{f}_{\p,t} \circ T^n \textup{d}\mu_{\p,t}.$$
Therefore,
\begin{eqnarray*}
\sigma_{\p,t}^2(f_{\p,t}) &=& \int \tilde{f}_{\p,t}^2 \textup{d}\mu_{\p,t} + 2\sum_{n=1}^{\infty} \int \tilde{f}_{\p,t} \cdot \tilde{f}_{\p,t}\circ T^n \textup{d}\mu_{\p,t} \\
&=& \int \tilde{f}_{\p,t}^2 \textup{d}\mu_{\p,t} + 2\sum_{n=1}^{\infty} \int \m_{\p,t}^n(\tilde{f}_{\p,t} \cdot \tilde{f}_{\p,t} \circ T^n)\textup{d}\mu_{\p,t} \\
&=& \int \tilde{f}_{\p,t}^2 \textup{d}\mu_{\p,t} +2 \sum_{n=1}^{\infty} \int \tilde{f}_{\p,t} \cdot \m_{\p,t}^n(\tilde{f}_{\p,t}) \textup{d}\mu_{\p,t} \\
&=& \int \tilde{f}_{\p,t}^2 \textup{d}\mu_{\p,t}
\end{eqnarray*}
since $\m_{\p,t}^n(\tilde{f}_{\p,t})=0$ for all $n \in \mathbb{N}$. 
\end{proof}

Now that we have managed to find a cohomologous function $\tilde{f}_{\p,t} \sim f_{\p,t}$ with the property that $\sigma^2_{\p,t}(f_{\p,t})= \int \tilde{f}_{\p,t}^2\textup{d}\mu_{\p,t}$, we can shift our focus to estimating $\int \tilde{f}^2_{\p,t}\textup{d}\mu_{\p,t}$ in a uniform way.

Let $I=[\frac{1}{8}, \frac{1}{4}]$ and notice that for all $t \in I$,
\begin{eqnarray}
\beta_{\p}(t) \geq \beta_{\p}\left(\frac{1}{4}\right)= \frac{\beta_{\p}(\frac{1}{4})}{1-\frac{1}{4}} \cdot \left(1-\frac{1}{4}\right) \geq \frac{3}{4} |\beta_{\p}^{\prime}(1)| \geq \frac{9}{16} \label{916} \end{eqnarray}
since $\beta_{\p}$ is convex and $\dim \mu_{\p} \geq \frac{3}{4}.$  Let $\mathcal{Z}$ be a finite set of periodic points of $T$. Suppose there exist constants $c_1$ and $c_2$ such that for all $\p$ and $t \in I$,
\begin{enumerate}
\item there exists a periodic point $z \in \mathcal{Z}$ of period $n$ such that $\frac{1}{n}|S_n \tilde{f}_{\p,t}(z)|\geq c_1$,
\item $[\tilde{f}_{\p,t}]_{\alpha}\leq c_2 $.
\end{enumerate}

Then we can bound $\int \tilde{f}^2_{\p,t}\textup{d}\mu_{\p,t}$ from below by a `strip' of the integral which is determined by an interval centred at an appropriate point $z'$ in the orbit of $z=\Pi(\i)$ for which $|\tilde{f}_{\p,t}(z')| \geq c_1$. We simply need to make the interval width sufficiently small so that $\tilde{f}_{\p,t}$ remains large within the interval, which we can do by using the H\"older properties of $\tilde{f}_{\p,t}$. In particular if $m$ is large enough that $\alpha^m \leq \frac{c_1}{2c_2}$ then for any $y \in \I_{\i|_m}$,
$$|\tilde{f}_{\p,t}(y)-\tilde{f}_{\p,t}(z)| \leq \frac{c_1}{2}$$
so it follows that for all $y \in \I_{\i|_m}$, $|\tilde{f}_{\p,t}(y)| \geq \frac{c_1}{2}$. Therefore
\begin{eqnarray}
\label{strategy}
\sigma_{\p,t}^2(f_{\p,t})= \int \tilde{f}^2_{\p,t}\textup{d}\mu_{\p,t} \geq \frac{c_1^2}{4} \mu_{\p,t}(\I_{\i|_m}).
\end{eqnarray}

Therefore we see that a uniform lower bound on $\sigma_{\p,t}^2(f_{\p,t})$ depends on the following three lemmas. In what follows each statement holds uniformly for all $\p$ that satisfy Hypothesis \ref{hyp} and all $t \in I$.

\begin{lma} \label{lemma periodic}
Given $\i \in \Sigma^{\ast}$ let $z_{\i}$ denote the periodic point for $T$ given by $z_{\i}=\Pi((\i)^{\infty})$. There exists a uniform constant $c_1$ independent of $\p$ and $t$ such that for any $t$ and $\p$, there exists $z \in \{z_1, z_2, z_{12}\}$ for which
$$\left|\frac{1}{2}S_{2} \tilde{f}_{\p, t}(z)\right| \geq c_1.$$
Moreoever, for any $\p$ which satisfies $p_1> \psi$, $|f_{\p,t}(z_1)| \geq c_1$.
\end{lma}

\begin{lma}
The function $U_{\p,t}\in \lip$ for all $t$ and $\p$. Moreover, there exists a uniform constant $c_2$ such that $[\tilde{f}_{\p,t}]_{\alpha} \leq c_2$. 
\label{lemma r}
\end{lma}

\begin{lma} \label{measure lemma}
There exists $c_3>0$ such that for any $\p$ and $t$, 
$$c_3^{-1} \frac{(p_{i_1}\cdots p_{i_n})^t}{|T^{\prime}(z) \cdots T^{\prime}(T^{n-1}z)|^{\beta_{\p}(t)}} \leq \mu_{\p,t}(\I_{i_1\ldots i_n}) \leq c_3 \frac{(p_{i_1}\cdots p_{i_n})^t}{|T^{\prime}(z) \cdots T^{\prime}(T^{n-1}z)|^{\beta_{\p}(t)}}$$
for any $n \in \N$, $i_1\ldots i_n \in \N^n$ and $z \in \I_{i_1\ldots i_n}$.
\end{lma}

In particular, if Lemmas \ref{lemma periodic} - \ref{measure lemma} hold then for each $\p$ and $t $ one can find $z \in \{z_1, z_2, z_{12}, z_{21}\}$ for which $|f_{\p,t}(z)| \geq c_1$. Therefore, by fixing $m$ sufficiently large that $\alpha^m \leq \frac{c_1}{2c_2}$ it follows that
$$\sigma_{\p,t}^2(f_{\p,t}) \geq \frac{c_1^2}{4}c_3^{-1}\frac{\epsilon^{\frac{m}{4}}}{9^m}$$
for all $t \in I$ and $\p$ that satisfy Hypothesis \ref{hyp}.

\subsection{Proof of Lemma \ref{lemma periodic}} \label{periodic}

We begin by proving Lemma \ref{lemma periodic}. Essentially this boils down to two key observations. Firstly observe that $S_n \tilde{f}_{\p,t}(z)=S_n f_{\p,t}(z)$ for any periodic point $z=T^nz$ since $f_{\p,t}$ and $\tilde{f}_{\p,t}$ are cohomologous. Secondly observe that by the non-linearity of $T$, $-\log|T^{\prime}|$ is not locally constant whereas $f_{\p}$ is. In particular, this means that
$$\left|\log \frac{T^{\prime}(z_1)T^{\prime}(z_2)}{T^{\prime}(z_{12})T^{\prime}(z_{21})}\right|\neq 0.$$ 

\vspace{5mm}  

\noindent \emph{Proof of Lemma \ref{lemma periodic}.} 
Fix $t$ and $\mathbf{p}= (p_1, p_2, \ldots)$. Recall that by the convexity of $\beta_{\p}$, $|\beta_{\p}^{\prime}(t)| \geq |\beta_{\p}^{\prime}(1)|=\dim \mu_{\p} >\frac{3}{4}$. Put 
$$c_{11}= \frac{1}{8} \left|\log \frac{T^{\prime}(z_1)T^{\prime}(z_2)}{T^{\prime}(z_{12})T^{\prime}(z_{21})}\right|>0.$$
Without loss of generality we can assume that both 
\begin{eqnarray}
|f_{\p,t}(z_1)|=|-\beta_{\p}^{\prime}(t)\log|T^{\prime}(z_1)|+\log p_1| < c_{11}
\label{z1}
\end{eqnarray}
and 
\begin{eqnarray}
|f_{\p,t}(z_2)|=|-\beta_{\p}^{\prime}(t)\log|T^{\prime}(z_2)|+\log p_{2}| < c_{11}
\label{z2}
\end{eqnarray}
since otherwise we are done. We will show that this forces $|\frac{1}{2}S_2 f_{\p,t}(z_{12})|>c_{11}$, which will complete the proof.

By (\ref{z1}) and (\ref{z2}) it follows that
$$\frac{1}{2}|-\beta_{\p}^{\prime}(t)\log |T^{\prime}(z_1)T^{\prime}(z_2)| + \log p_1p_{2}| \leq c_{11}.$$
Moreover
\begin{eqnarray*}
4|\beta_{\p}^{\prime}(t)|c_{11}&=&\frac{|\beta_{\p}^{\prime}(t)|}{2}\left|\log \frac{T^{\prime}(z_1)T^{\prime}(z_2)}{T^{\prime}(z_{12})T^{\prime}(z_{21})}\right| \\
 &\leq& \frac{1}{2}\left|-\beta_{\p}^{\prime}(t)\log |T^{\prime}(z_1)T^{\prime}(z_2)|+\log p_1p_{2}\right| \\
& & +\frac{1}{2}\left|-\beta_{\p}^{\prime}(t)\log |T^{\prime}(z_{12})T^{\prime}(z_{21})|+\log p_1p_{2}\right| \\
&\leq &\frac{1}{2}\left|-\beta_{\p}^{\prime}(t)\log |T^{\prime}(z_{12})T^{\prime}(z_{21})|+\log p_1p_{2}\right|+c_{11}.
\end{eqnarray*}
Therefore
\begin{eqnarray*}
\frac{1}{2}\left|-\beta_{\p}^{\prime}(t)\log |T^{\prime}(z_{12})T^{\prime}(z_{21})|+\log p_1p_{2}\right| \geq 4|\beta_{\p}^{\prime}(t)|c_{11}- c_{11} \geq 2c_{11}
\end{eqnarray*}
where the final inequality is because $|\beta^{\prime}_{\p}(t)| \geq \frac{3}{4}$. 

Next, put $c_{12}= \frac{1}{2} \log|T^{\prime}(z_1)|$. Recall that by definition of $\psi$ in (\ref{psi}), if $p_1> \psi$ this implies $\log p_1 \geq -\frac{1}{4} \log|T^{\prime}(z_1)|$. Therefore, since $-\beta_{\p}^{\prime}(t) \geq \frac{3}{4}$ it follows that
$$|f_{\p,t}(z_1)| \geq \frac{3}{4} \log|T^{\prime}(z_1)|-\frac{1}{4} \log|T^{\prime}(z_1)|= \frac{1}{2}\log|T^{\prime}(z_1)|=c_{12}.$$

Finally, putting $c_1=\min\{c_{11}, c_{12}\}$ we complete the proof.
\qed

\subsection{Proof of Lemma \ref{lemma r}}

In this section we will prove Lemma \ref{lemma r}. By \cite[Theorem 2.4.6]{mu} we know that for each $\p$ and $t$ there exist constants $c_{\p,t}>0$, $0<\rho_{\p,t}<1$ such that for all $f \in \mathcal{H}_{\alpha}$ with $\mu_{\p,t}(f)=0$,
$$\norm{\m_{\p,t}^nf}_{\alpha} \leq c_{\p,t}\rho_{\p,t}^n \norm{f}_{\alpha}.$$
We would like to prove a uniform (in $\p$ and $t$) version of the above property. In fact we will work with the Lipschitz norm instead, and show that we can choose uniform $c_4>0$, $0<\rho<1$ such that for all $\p$ and $t$ and  $f \in \lip$ with $\mu_{\p,t}(f)=0$,
\begin{eqnarray}
\norm{\m_{\p,t}^nf}_{0,1} \leq c_4\rho^n \norm{f}_{0,1}.
\label{c3 thing}
\end{eqnarray}
To do this we will make use of `Hilbert-Birkhoff cone theory' \cite{liverani, viana} which provides technology that yields particularly explicit estimates for the rate of decay of norms under transfer operators, which will allow us to verify that a uniform property such as (\ref{c3 thing}) holds. The result will then follow by obtaining upper bounds on $\norm{\m_{\p,t}f_{\p,t}}_{0,1}$ and $\norm{\m_{\p,t}^2f_{\p,t}}_{0,1}$. 

We begin this section by summarising the tools from Hilbert-Birkhoff cone theory and how these can be applied to transfer operators. For more details the reader is directed to \cite{liverani, viana}. For $a>0$ define
\begin{eqnarray*}
\mathcal{C}_a &=& \left\{ w \in \cont: w \geq 0 \textnormal{  and  } w(x) \leq e^{a|x-y|} w(y)\right\}.
\label{cone}
\end{eqnarray*}
Then $\c$ is a closed convex cone; this means that $\lambda w \in \c$  and $w_1 + w_2 \in \c$ for all $\lambda >0$ and all $w, w_1, w_2 \in \c$. We can define a partial ordering $\preceq$ on $V$ by
\begin{eqnarray*}
\begin{array}{ccc}
v \preceq w & \Leftrightarrow & w-v \in \c \cup \{0\}.
\end{array}
\end{eqnarray*}
Moreover, using this partial ordering one can define the \emph{projective metric} $\Theta$ on $\c$; we will not actually require an explicit characterisation of this metric but it is defined and discussed in \cite[Proposition 2.2 and Example 2.3]{viana}. The following proposition follows from \cite[Propositions 2.3 and 2.5]{viana}.

\begin{prop}\label{contract prop}
Let $L:C([0,1]) \to C([0,1])$ be a linear operator and $\c$ be the cone as defined in (\ref{cone}), equipped with the projective metric $\Theta$. Suppose there exists $a>0$ and $\lambda \in (0,1)$ such that $L(\c) \subset \mathcal{C}_{\lambda a}$. Then 
\begin{enumerate}
\item 
\begin{equation}
D=\sup_{v, w \in \c} \Theta(L(v), L(w)) < \infty, \label{diam}
\end{equation}
\item there exists $r \in (0,1)$ that depends only on $D$ such that for all $v, w \in \c$,
$$ \Theta(L(v), L(w)) \leq r\Theta(v,w).$$

\end{enumerate}
\end{prop} 

The following is an easy modification of \cite[Lemma 1.3]{liverani}.

\begin{prop} \label{transfer}
Let $\norm{\cdot}_1$, $\norm{\cdot}_2$ be two norms on $C([0,1])$ and consider the cone $\c$ which induces the partial ordering $\preceq$. Suppose there exists $C \geq 1$ such that for all $f, g \in C([0,1])$
\begin{eqnarray*}
\begin{array}{ccc}
-f \preceq g \preceq f & \Rightarrow & \norm{g}_1 \leq \norm{f}_1 \\
& & \norm{g}_2 \leq C\norm{f}_2.
\end{array}
\end{eqnarray*}
Then given any $f, g \in \c$ for which $\norm{f}_1=\norm{g}_1$,
$$\norm{f-g}_2 \leq C^2(e^{\Theta(f,g)}-1)\norm{f}_2.$$
\end{prop}

We also note that it is easy to check that $-f \preceq g \preceq f$ implies that $\norm{g}_{\infty} \leq \norm{f}_{\infty}$ and $\norm{g}_{L^1(m)} \leq \norm{f}_{L^1(m)}$ for any measure $m$ on $[0,1]$. Additionally one can check that $-f \preceq g \preceq f$ implies that $\norm{g}_{0,1} \leq (1+a)^2 \norm{f}_{0,1}$. 

We will now apply Proposition \ref{contract prop} to the operator $\m_{\p,t}^2$ to deduce that it strictly contracts the projective metric $\Theta$ (with the view to later combine this with Proposition \ref{transfer} in order to prove (\ref{c3 thing})). The following lemma will also provide us with uniform regularity properties for the fixed points $h_{\p,t}$.

\begin{lma}
There exists $a>0$, $D< \infty$ and $r \in (0,1)$ such that for all $v, w \in \c$, 
\begin{eqnarray}
\sup_{v, w \in \c} \Theta(\m^2_{\p,t}(v), \m^2_{\p,t}(w)) \leq D \label{diamm}
\end{eqnarray}
and
\begin{eqnarray} \Theta(\m^2_{\p,t}(v), \m^2_{\p,t}(w)) \leq r\Theta(v,w). \label{contract} \end{eqnarray}
Moreoever, for all $x, y \in [0,1] \setminus \Q$ and all $\p$, $t$
\begin{eqnarray}
\exp(-a|x-y|) \leq \frac{h_{\p,t}(x)}{h_{\p,t}(y)} \leq \exp(a|x-y|). \label{fp} \end{eqnarray}
\label{fixed pt}
\end{lma}

\begin{proof}
We begin by proving that the analogues of (\ref{diamm}) and (\ref{contract}) hold for $\l_{\p,t}^2$ for some $a_0$, $r_0$ and $D_0$. 

Since $f_{\p}$ is locally constant,
$$[g_{\p,t}]_{\alpha}=[\beta_{\p}(t)\log|T^{\prime}|]_{\alpha} \leq [\log|T^{\prime}|]_{\alpha}$$
so there exists $\kappa< \infty$ such that $[g_{\p,t}]_{\alpha} \leq \kappa$ for all $\p$ and $t$. 

Let $a_0>0$, $w \in \mathcal{C}_{a_0}$ and $x, y \in [0,1]$. Recall that for all $x$, $|(T^2)^{\prime}(x)| \geq \frac{1}{\alpha^2}=\frac{9}{4}$. In particular, this means that any local inverse branch of $T^2$ must be contracting by $\alpha^2$. Thus
\begin{eqnarray*}
(\l^2_{\p,t} w)(x) &=&\sum_{\n \in \mathbb{N}^2} \exp(g^2_{\p,t}(T_{\n}^{-1}x))w\left(T_{\n}^{-1}x\right)\\
&\leq&\sum_{\n \in \mathbb{N}^2} \exp(g^2_{\p,t}(T_{\n}^{-1}y))w\left(T_{\n}^{-1}y\right)\exp((2\kappa +a_0)|T_{\n}^{-1}x- T_{\n}^{-1}y|)  \\
&\leq& \sum_{\n \in \mathbb{N}^2} \exp(g^2_{\p,t}(T_{\n}^{-1}y))w\left(T_{\n}^{-1}y\right)\exp(\alpha^2(2\kappa +a_0)|x-y|).
\end{eqnarray*}
Choose $\alpha^2< \lambda_0 <1$ and $a_0 \geq \frac{2\alpha^2\kappa}{\lambda_0 - \alpha^2}$. Then it follows that
\begin{eqnarray}
(\l^2_{\p,t} w)(x) &\leq & (\l^2_{\p,t} w)(y) \exp(a_0\lambda_0 |x-y|).
\label{forcont}
\end{eqnarray}
Clearly $\l_{\p,t}^2w \geq 0$ and $\l_{\p,t}^2 w \in C([0,1])$. Therefore $\l_{\p,t}^2 \mathcal{C}_{a_0} \subset \mathcal{C}_{\lambda_0a_0}$.
Therefore by Proposition \ref{contract prop} there exists $D_0< \infty$ such that $\sup_{v,w \in \c}(\Theta(L(v), L(w))) \leq D_0$ and there exists some $r_0\in (0,1)$ which depends only on $D_0$ for which
\begin{eqnarray} \Theta(\l^2_{\p,t}(v), \l^2_{\p,t}(w)) \leq r_0\Theta(v,w) \label{l contract} \end{eqnarray}
for all $v, w \in \mathcal{C}_{a_0}$. In particular $r_0$ and $D_0$ are independent of $\p$ and $t$. 

Using (\ref{l contract}) we can prove (\ref{fp}) for $a=a_0$. Let $N \in \N$ and consider integers $m, n \geq N$. Using (\ref{l contract}) we can write
$$\Theta(\l_{\p,t}^{2n+k} \one, \l_{\p,t}^{2m+k} \one) \leq r_0^N\Theta(\l_{\p,t}^{2(n-N)+k} \one, \l_{\p,t}^{2(m-N)+k}\one) \leq D_0r_0^N$$
for each $k \in \{0,1\}$. Let $L^1=L^1(\tilde{\mu}_{\p,t})$. Since $\norm{\l_{\p,t}^j \one}_{L^1}= \norm{\one}_{L^1}$ for all $j \in \mathbb{N}$, we can apply Proposition \ref{transfer} to the norms $\norm{\cdot}_1=\norm{\cdot}_{L^1}$ and $\norm{\cdot}_2= \norm{\cdot}_{\infty}$ to deduce that for all $n, m \geq N$,
\begin{eqnarray*}
\norm{\l_{\p,t}^{2n+k} \one- \l_{\p,t}^{2m+k} \one}_{\infty} &\leq& \exp(\Theta(\l_{\p,t}^{2n+k} \one,\l_{\p,t}^{2m+k} \one))-1 \\
&\leq& \exp(D_0r_0^N)-1 \\
&\leq& D_0\exp(D_0)r_0^N.
\end{eqnarray*}
This implies that $\l_{\p,t}^n \one$ is a Cauchy sequence in the uniform norm $\norm{\cdot}_{\infty}$. Thus the limit $\lim_{n \to \infty} \l_{\p,t}^n \one \in \mathcal{C}_{a_0}$ and is a fixed point of $\l_{\p,t}$. In particular, since the fixed point is unique, this means that $h_{\p,t}=\lim_{n \to \infty} \l_{\p,t}^n \one$ and therefore $h_{\p,t}$ satisfies (\ref{fp}) for $a=a_0$. We now use this fact to prove (\ref{diamm}) and (\ref{contract}).

Let $a_1>0$. Then
\begin{eqnarray*}
& &(\m^2_{\p,t} w)(x) =h_{\p,t}^{-1}(x)\sum_{\n \in \mathbb{N}^2} \exp(g^2_{\p,t}(T_{\n}^{-1}x))w\left(T_{\n}^{-1}x\right)h_{\p,t}\left(T_{\n}^{-1}x\right)\\
& &\leq h_{\p,t}^{-1}(x)\sum_{\n \in \mathbb{N}^2} \exp(g^2_{\p,t}(T_{\n}^{-1}y))w\left(T_{\n}^{-1}y\right)h_{\p,t}(T_{\n}^{-1}y)\exp(((2\kappa+a_0+a_1)|T_{\n}^{-1}x- T_{\n}^{-1}y|))  \\
& &\leq h_{\p,t}^{-1}(y) \sum_{\n \in \mathbb{N}^2} \exp(g^2_{\p,t}(T_{\n}^{-1}y))w\left(T_{\n}^{-1}y\right)h_{\p,t}(T_{\n}^{-1}y)\exp(a_0+\alpha^2(2\kappa+ a_0+a_1)|x-y|) .
\end{eqnarray*}
Choose $\alpha^2< \lambda_1 <1$ and $a_1 \geq \frac{a_0+2\alpha^2\kappa+\alpha^2 a_0}{\lambda_1 - \alpha^2}$. Then it follows that
\begin{eqnarray}
(\m^2_{\p,t} w)(x) \leq  (\m^2_{\p,t} w)(y) \exp(a_1\lambda_1 |x-y|).
\label{forcont2}
\end{eqnarray}
Fix $a= \max\{1, a_1\}$. Clearly $\m_{\p,t}^2w \geq 0$ and $\m_{\p,t}^2 w \in C([0,1])$. Therefore $\m_{\p,t}^2 \mathcal{C}_{a} \subset \mathcal{C}_{\lambda a}$. Thus by Proposition \ref{contract prop}, (\ref{contract}) holds. Moreover since $a>a_0$, $h_{\p,t} \in \c$ and so (\ref{fp}) holds.
\end{proof}

Next we obtain a uniform upper bound on the operator norm of $\m_{\p,t}$, when restricted to the cone $\c$.

\begin{lma}
There exists $A>0$ such that for all $f \in \c$ and all $\p$, $t$,
$$\norm{\m_{\p,t}^{2n}f}_{0,1} \leq A\norm{f}_{0,1}.$$
\label{op norm}
\end{lma}

\begin{proof}
Firstly, we can immediately see that $\norm{\m_{\p,t}^kf}_{\infty} \leq \norm{f}_{\infty}$ for all $k \in \N$. Next, since $f \in \c$, by Lemma \ref{fixed pt} it follows that $\m_{\p,t}^{2n} f \in \c$ as well and therefore setting $F=\m_{\p,t}^{2n} f$,
$$-(e^{a|x-y|}-1)F(x) \leq F(x)-F(y) \leq (e^{a|x-y|}-1)F(y)$$
for all $x, y \in [0,1]$ which implies that
$$|F(x)-F(y)| \leq ae^a \norm{F}_{\infty} |x-y|$$
that is, $F$ is Lipschitz with Lipschitz constant $[F]_1 \leq  ae^a \norm{F}_{\infty} $. Thus $[\m_{\p,t}^{2n}f]_1 \leq ae^a\norm{\m_{\p,t}^{2n}f}_{\infty} \leq ae^a\norm{f}_{\infty}$. 
\end{proof}

Now using lemmas \ref{fixed pt} and \ref{op norm} we can apply Proposition \ref{transfer} to the operator $\m_{\p,t}^2$ to deduce that (\ref{c3 thing}) holds.

\begin{lma}
There exist constants $c_4>0$ and $0<\rho<1$ such that for all $\p$, $t$ and $f \in \lip$ with $\mu_{\p,t}(f)=0$,
$$\norm{\m_{\p,t}^nf}_{0,1} \leq c_4\rho^n \norm{f}_{0,1}.$$
\label{decay}
\end{lma}

\begin{proof}
Let $f \in \lip$ for which $\mu_{\p,t}(f)=0$. If $f$ is constant, $f=0$ since its integral is 0 and thus the result follows trivially. If $f$ is not constant, $\norm{f}_{0,1} >0$. Let $f_1$ and $f_2$ be the positive and negative parts of $f$ respectively, so that $f=f_1-f_2$ with $f_1, f_2 \geq 0$. We can guarantee that they belong to a cone by adding a constant. In particular, $f_i+ \norm{f}_{0,1} \in \mathcal{C}_1$ for each $i$ since 
\begin{eqnarray*}
\frac{f_i(x)+ \norm{f}_{0,1}}{f_i(y)+ \norm{f}_{0,1}} &=& \exp \left(\log \left( \frac{f_i(x)+ \norm{f}_{0,1}}{f_i(y)+ \norm{f}_{0,1}}\right)\right) \\
&=& \exp \left(\log \left( \frac{f_i(x)-f_i(y)}{f_i(y)+ \norm{f}_{0,1}}+1\right)\right) \\
&\leq& \exp \left(\log \left( \frac{\norm{f}_{0,1}|x-y|}{f_i(y)+ \norm{f}_{0,1}}+1\right)\right) \\
&\leq& \exp \left( \frac{\norm{f}_{0,1}|x-y|}{f_i(y)+ \norm{f}_{0,1}}\right) \\
&\leq& \exp \left( \frac{\norm{f}_{0,1}|x-y|}{\norm{f}_{0,1}}\right) \\
&=& \exp (|x-y|)
\end{eqnarray*}
where the fourth line follows because $\log(1+z) \leq z$ for any $z>-1$. Denote $\eta= \norm{f}_{0,1}$. Then $f_i + \eta \in \c$. Then since $\mu_{\p,t}(f_1)= \mu_{\p,t}(f_2)$ we have
\begin{eqnarray*}
\norm{\m_{\p,t}^{2n}f}_{0,1} &=& \norm{\m_{\p,t}^{2n}(f_1+\eta)-\m_{\p,t}^{2n}(f_2+\eta)}_{0,1} \\
&\leq& \norm{\m_{\p,t}^{2n}(f_1+\eta)-\mu_{\p,t} (f_1+\eta)}_{0,1} +\norm{\m_{\p,t}^{2n}(f_2+\eta)-\mu_{\p,t}(f_2+\eta)}_{0,1}.
\end{eqnarray*}

Denoting $L^1=L^1(\mu_{\p,t})$, we can apply Proposition \ref{transfer} for $\norm{\cdot}_1=\norm{\cdot}_{L^1}$ and $\norm{\cdot}_2=\norm{\cdot}_{0,1}$ to obtain
\begin{eqnarray*}
\norm{\m_{\p,t}^{2n}f}_{0,1} &\leq& (1+a)^2(\exp(\Theta(\m_{\p,t}^{2n}( f_1+\eta), \mu_{\p,t}( f_1+\eta)\one))-1)\norm{\m_{\p,t}^{2n} (f_1+\eta)}_{0,1} \\
& & + (1+a)^2(\exp(\Theta(\m_{\p,t}^{2n} (f_2+\eta), \mu_{\p,t}(f_2+\eta)\one))-1)\norm{\m_{\p,t}^{2n} (f_2+\eta)}_{0,1} .
\end{eqnarray*}

Next we can apply (\ref{contract}) to get
\begin{eqnarray*}
\norm{\m_{\p,t}^{2n}f}_{0,1} &\leq& (1+a)^2(\exp(r^n\Theta(f_1+\eta, \mu_{\p,t}(f_1+\eta)\one))-1)\norm{\m_{\p,t}^{2n} (f_1+\eta)}_{0,1} \\
& & +(1+a)^2(\exp(r^n\Theta( f_2+\eta, \mu_{\p,t}(f_2+\eta)\one))-1)\norm{\m_{\p,t}^{2n} (f_2+\eta)}_{0,1}\\
&\leq& (1+a)^2(\exp(r^nD)-1)A(\norm{f_1+\eta}_{0,1}+\norm{f_2+\eta}_{0,1}) \\
&\leq& (1+a)^2Dr^n\exp(Dr^n) (\norm{f_1}_{0,1}+\norm{f_2}_{0,1}+2\eta)\\
&\leq& 4(1+a)^2ADe^Dr^n\norm{f}_{0,1}
\end{eqnarray*}
where $A$ is the uniform constant from Lemma \ref{op norm}.
\end{proof}

Before we can use the above result to prove Lemma \ref{lemma r}, we need uniform bounds on $\norm{\m_{\p,t}f_{\p,t}}_{0,1}$ and $\norm{\m_{\p,t}^2f_{\p,t}}_{0,1}$ for all $\p$ that satisfy Hypothesis \ref{hyp} and $t \in I$. Observe that for $t \in I$, $|\beta_{\p}^{\prime}(t)|, |p_n^t \log p_n| \leq 8$. To see that this holds for $|p_n^t \log p_n|$, define $\alpha_t(x)= x^t\log x$ for $x \in [0,1]$. Differentiating with respect to $x$ we obtain
$$\frac{\textup{d}}{\textup{d}x} (\alpha_t(x))= tx^{t-1}\log x+x^{t-1}= x^{t-1}(t\log x+1).$$
Clearly the only turning point in $[0,1]$ is $x= e^{-\frac{1}{t}}$ and since $\alpha_t(0)=\alpha_t(1)=0$ this is a local minimum for $\alpha_t$, that is, a local maximum for $|x^t\log x|$. Moreover, for $t \geq \delta>0$, 
$$\alpha_t(e^{-\frac{1}{t}})= e^{-1}\log e^{-\frac{1}{t}}=-\frac{1}{t}e^{-1} \geq -\frac{1}{\delta}e^{-1}= \alpha_{\delta}(e^{-\frac{1}{\delta}}).$$
Therefore, 
$$|x^t\log x| \leq |\alpha_t(e^{-\frac{1}{t}})| \leq |\alpha_{\delta}(e^{-\frac{1}{\delta}})| = \frac{1}{\delta}e^{-1} \leq \frac{1}{\delta}.$$
The claim follows because $I= [\frac{1}{8}, \frac{1}{4}]$.

\begin{lma}
There exists a constant $c_5>0$ such that for all $\p$ that satisfies Hypothesis \ref{hyp}, all $t \in I$ and $k \in \{1,2\}$
\begin{eqnarray*}
\norm{\m_{\p,t}^kf_{\p,t}}_{0,1} \leq c_5. 
\end{eqnarray*} \label{e1e2}
\end{lma}

\begin{proof}
Observe that
$$\m_{\p,t}f_{\p,t}(x)= h_{\p,t}^{-1}(x) \sum_{n \in \N} \frac{-p_n^t \beta_{\p}^{\prime}(t)\log|T^{\prime}(T_n^{-1}(x))| +p_n^t\log p_n}{|T^{\prime}(T_n^{-1}x)|^{\beta_{\p}(t)}} h_{\p,t}(T_n^{-1}x).$$
By (\ref{fp}) and Proposition \ref{bd}, there exists a uniform constant $C$ which is independent of $\p$, $t$ and $x$ such that
$$|\m_{\p,t}f_{\p,t}(x)| \leq C\sum_{n \in \N} \frac{\log n}{n^{2\beta_{\p}(t)}}.$$
Thus we get a uniform upper bound for $\norm{\m_{\p,t}f_{\p,t}}_{\infty}$ and $\norm{\m_{\p,t}^2f_{\p,t}}_{\infty}$ by recalling that $\norm{\m_{\p,t}^2f_{\p,t}}_{\infty}\leq \norm{\m_{\p,t}f_{\p,t}}_{\infty}$ and because $\beta_{\p}(t) \geq \frac{9}{16}$ by (\ref{916}).

To obtain the desired bound on $[\m_{\p,t}f_{\p,t}]_1$, we write $\m_{\p,t}f_{\p,t}$ in the form
$$\m_{\p,t}f_{\p,t}(x)= \sum_{n \in \N} h_n(x) u_n(x)$$
where $h_n(x)= h_{\p,t}^{-1}(x)h_{\p,t}((x+n)^{-1})$ and $u_n$ is given by
$$u_n(x)= \frac{2p_n^t\beta_{\p}^{\prime}(t)\log(x+n)+p_n^t\log p_n}{(x+n)^{2\beta_{\p}(t)}}.$$
Observe that since $h_{\p,t} \in \c$, we can use the same arguments as in the proof of Lemma \ref{op norm} to deduce that
$|h_{\p,t}(x)-h_{\p,t}(y)| \leq ae^a\norm{h_{\p,t}}_{\infty}|x-y|$. Therefore using (\ref{fp}) and the inequality
$$|h_n(x)-h_n(y)| \leq \norm{h_{\p,t}^{-1}}_{\infty} |h_{\p,t}((x+n)^{-1})-h_{\p,t}((y+n)^{-1})|+ \norm{h_{\p,t}}_{\infty}|h_{\p,t}^{-1}(x)-h_{\p,t}^{-1}(y)|$$
we obtain
\begin{eqnarray*}
|h_n(x)-h_n(y)| &\leq& ae^a\norm{h_{\p,t}}_{\infty}\norm{h_{\p,t}^{-1}}_{\infty}|x-y| +ae^a\norm{h_{\p,t}}_{\infty}^2\norm{h^{-1}_{\p,t}}_{\infty}^2|x-y|\\
&\leq& 2ae^{3a}|x-y|.
\end{eqnarray*}
Using this, (\ref{fp}) and the inequality
$$|h_n(x)u_n(x)-h_n(y)u_n(y)| \leq \norm{h_n}_{\infty} |u_n(x)-u_n(y)|+ \norm{u_n}_{\infty}|h_n(x)-h_n(y)|,$$
we obtain
\begin{eqnarray}
|\m_{\p,t}f_{\p,t}(x)-\m_{\p,t}f_{\p,t}(y)| \leq \sum_{n \in \N} e^a|u_n(x)-u_n(y)|+\sum_{n \in \N} 2ae^{3a}\norm{u_n}_{\infty}|x-y|. \label{new}
\end{eqnarray}
There exists a uniform constant $C$ which is independent of $\p$, $t$ and $x$ such that 
$$u_n(x) \leq C \frac{\log n}{n^{2\beta_{\p}(t)}}$$
therefore the second sum in (\ref{new}) is uniformly bounded for all $\p$ and $t$. To verify that the first sum is bounded by a constant multiple of $|x-y|$, observe that
$$\frac{\textup{d}}{\textup{d}x} u_n(x)= \frac{2p_n^t(\beta_{\p}^{\prime}(t)-2\beta_{\p}^{\prime}(t)\beta_{\p}(t) \log(x+n)-\beta_{\p}(t) \log p_n)}{(x+n)^{2\beta_{\p}(t)+1}}.$$
As before, there exists some uniform constant $C$ which is independent of $\p$, $t$ and $x$ such that
$$\frac{\textup{d}}{\textup{d}x} u_n(x) \leq C \frac{\log n}{n^{2\beta_{\p}(t)+1}}$$
from which it follows that the first sum is also uniformly bounded by some constant multiple of $|x-y|$ which is independent of $\p$ and $t$. We can bound $|\m^2_{\p,t}f_{\p,t}(x)-\m^2_{\p,t}f_{\p,t}(y)|$ similarly, thus the result follows.

\end{proof}

We are now in a position to prove Lemma \ref{lemma r}.

\vspace{5mm}

\noindent \emph{Proof of Lemma \ref{lemma r}.} For all $\p$ that satisfy Hypothesis \ref{hyp} and $t \in I$ and $n \geq 1$,
\begin{eqnarray*}
[\m_{\p,t}^{2n} f_{\p,t}]_{1} &\leq& \norm{\m_{\p,t}^{2n} f_{\p,t}}_{0,1} \\
&\leq& c_4 \rho^{n-1} \norm{\m^2_{\p,t}f_{\p,t}}_{0,1} \\
&\leq& c_4c_5 \rho^{n-1} 
\end{eqnarray*}
where the penultimate inequality follows by Lemma \ref{decay} and the last inequality follows by Lemma \ref{e1e2}. We can obtain an analogous upper bound for $[\m_{\p,t}^{2n+1} f_{\p,t}]_1$ for $n \geq 1$.

Therefore, $[U_{\p,t}]_{1}$ is uniformly bounded in $\p$ and $t$, which in turn implies that $[U_{\p,t}]_{\alpha}$ is uniformly bounded in $\p$ and $t$. Since
$$[f_{\p,t}]_{\alpha}=[\beta_{\p}^{\prime}(t)\log|T^{\prime}|]_{\alpha} \leq \left[\beta_{\p}^{\prime}\left(\frac{1}{8}\right)\log|T^{\prime}|\right]_{\alpha} \leq 8[\log|T^{\prime}|]_{\alpha}$$
the result follows. \qed

\subsection{Proof of Lemma \ref{measure lemma}}

In this section we investigate the Gibbs properties of $\mu_{\p,t}$ and thus prove Lemma \ref{measure lemma}. Consequently this also allows us to deduce that $\int \log|T^{\prime}|d\mu_{\p,t}$,  which appears in the expression for $\beta_{\p}^{\prime\prime}(t)$ in (\ref{beta2}), is uniformly bounded above for any $\p$ that satisfies Hypothesis \ref{hyp} and all $t \in I$.

By definition, $\mu_{\p,t}$ is a Gibbs measure for the potential $g_{\p,t}$ and therefore we know that for each $\p$ and $t$ there exists a constant $0<C_{\p,t}< \infty$ such that for all $n \in \N$ and all $i_1 \ldots i_n \in \Sigma^{\ast}$,
\begin{equation}
C_{\p,t}^{-1} \frac{(p_{i_1}\cdots p_{i_n})^t}{|T^{\prime}(z) \cdots T^{\prime}(T^{n-1}z)|^{\beta_{\p}(t)}} \leq \mu_{\p,t}(\I_{i_1\ldots i_n}) \leq C_{\p,t} \frac{(p_{i_1}\cdots p_{i_n})^t}{|T^{\prime}(z) \cdots T^{\prime}(T^{n-1}z)|^{\beta_{\p}(t)}}.
\label{gibbs pt}
\end{equation}
We'll prove that in fact we can choose a uniform constant $c_3$ such that (\ref{gibbs pt}) becomes
\begin{equation}
c_3^{-1} \frac{(p_{i_1}\cdots p_{i_n})^t}{|T^{\prime}(z) \cdots T^{\prime}(T^{n-1}z)|^{\beta_{\p}(t)}} \leq \mu_{\p,t}(\I_{i_1\ldots i_n}) \leq c_3 \frac{(p_{i_1}\cdots p_{i_n})^t}{|T^{\prime}(z) \cdots T^{\prime}(T^{n-1}z)|^{\beta_{\p}(t)}}
\label{gibbs pt2}
\end{equation}
uniformly for all $\p$ and $t$.

\vspace{5mm}
\noindent \emph{Proof of Lemma \ref{measure lemma}.} Recall that $\tilde{g}_{\p,t}=g_{\p,t}+h_{\p,t}-h_{\p,t} \circ T$ (see Proposition \ref{rpf2}). Since $f_{\p}$ is locally constant, 
$$[g_{\p,t}]_{\alpha}=[-\beta_{\p}(t)\log|T^{\prime}|]_{\alpha} \leq [\log|T^{\prime}|]_{\alpha}$$
and therefore $[g_{\p,t}]_{\alpha}$ can be bounded above by a constant which is independent of $\p$ and $t$. Also if $x, y \in \I_{i_1\ldots i_n}$ then by Lemma \ref{fixed pt}
$$|\log h_{\p,t}(x)-\log h_{\p,t}(y)|= \left|\log \frac{ h_{\p,t}(x)}{ h_{\p,t}(y)}\right| \leq a|x-y| $$
and therefore $[h_{\p,t}]_{\alpha}$ can be bounded above by a constant which is independent of $\p$ and $t$. Therefore there exists $\tau>0$ such that $[\tilde{g}_{\p,t}]_{\alpha} \leq \tau$ for all $\p$ and $t$.

Now we can apply arguments similar to \cite{bowen}. Let $n \in \N$ and any $i_1\ldots i_n \in \N^n$. Then
\begin{eqnarray*}
\mu_{\p,t}(\I_{i_2,\ldots,i_n}) &=& \int \one_{\I_{i_2,\ldots,i_n}}(x) \textup{d}\mu_{\p,t}(x) \\
&=& \int \sum_{Ty=x} \one_{\I_{i_1, i_2,\ldots,i_n}}(y) \textup{d}\mu_{\p,t}(x) \\
&=& \int \sum_{Ty=x} \exp(\tilde{g}_{\p,t}(y))\one_{\I_{i_1, \ldots, i_n}}(y)\exp(-\tilde{g}_{\p,t}(y))\textup{d}\mu_{\p,t}(x) \\
&=& \int \m_{\p,t}(\one_{\I_{i_1, \ldots, i_n}}(x)\exp(-\tilde{g}_{\p,t}(x)))\textup{d}\mu_{\p,t}(x) \\
&=& \int_{\I_{i_1, \ldots, i_n}} \exp(-\tilde{g}_{\p,t}(x))\textup{d}\mu_{\p,t}(x)
\end{eqnarray*}
where the final line follows because $\m_{\p,t}^{\ast} \mu_{\p,t}=\mu_{\p,t}$.

Let $z \in \I_{i_1\ldots i_n}$. Then
$$\mu_{\p,t}(\I_{i_2, \ldots, i_n})\exp(\tilde{g}_{\p,t}(z)) \leq \exp(\alpha^n[\tilde{g}_{\p,t}]_{\alpha}) \mu_{\p,t}(\I_{i_1, \ldots, i_n})$$
so that
$$\frac{\mu_{\p,t}(\I_{i_1, \ldots, i_n})}{\mu_{\p,t}(\I_{i_2, \ldots, i_n})} \exp(-\tilde{g}_{\p,t}(z)) \geq \exp(-\alpha^n[\tilde{g}_{\p,t}]_{\alpha}).$$
Moreover, we can proceed to obtain the following sequence of inequalities
\begin{eqnarray*}
\frac{\mu_{\p,t}(\I_{i_2, \ldots, i_n})}{\mu_{\p,t}(\I_{i_3, \ldots, i_n})} \exp(-\tilde{g}_{\p,t}(Tz)) &\geq& \exp(-\alpha^{n-1}[\tilde{g}_{\p,t}]_{\alpha})\\ 
\vdots \\
\mu_{\p,t}(\I_{i_n}) \exp(-\tilde{g}_{\p,t}(T^{n-1}z)) &\geq& \exp(-\alpha[\tilde{g}_{\p,t}]_{\alpha}).
\end{eqnarray*}
Multiplying these all together we obtain
\begin{eqnarray}
\frac{\mu_{\p,t}(\I_{i_1, \ldots, i_n})}{\exp(S_n\tilde{g}_{\p,t}(z))} &\geq& \exp\left(-\frac{[\tilde{g}_{\p,t}]_{\alpha}}{1-\alpha}\right).
\label{gibbs tilde}
\end{eqnarray}
Now, 
\begin{eqnarray*}
S_n( \log h_{\p,t}-\log h_{\p,t} \circ T)(z)&=&  \log h_{\p,t}(z)-\log h_{\p,t}(Tz) \\
& & +  \log h_{\p,t}(Tz)-\log h_{\p,t}(T^2z) \\
& & \vdots \\
& & + \log h_{\p,t}(T^{n-1}z)-\log h_{\p,t}(T^nz) \\
&=& \log \frac{h_{\p,t}(z)}{h_{\p,t}(T^n z)} \geq -a.
\end{eqnarray*}
Plugging this into (\ref{gibbs tilde}) we obtain
\begin{eqnarray}
\frac{\mu_{\p,t}(\I_{i_1, \ldots, i_n})}{\exp(S_ng_{\p,t}(z))} &\geq& \exp\left(-\frac{[\tilde{g}_{\p,t}]_{\alpha}}{1-\alpha}-a\right) \geq \exp\left(-\frac{\tau}{1-\alpha}-a\right).
\label{gibbs gpt}
\end{eqnarray}
By rearranging this inequality and expanding the ergodic sum we obtain the desired lower bound. The upper bound follows by an analogous argument.
\qed

We can now deduce that $\int \log|T^{\prime}|d\mu_{\p,t}$ is uniformly bounded above for all $\p$ and $t$.

\begin{lma} 
There exists a uniform constant $L$ such that for all $\p$ that satisfy Hypothesis \ref{hyp} and $t\in I$, 
$$\int \log|T^{\prime}|d\mu_{\p,t} \leq L.$$
\label{e2 proof}
\end{lma}

\begin{proof}
By Lemma \ref{measure lemma}, $\mu_{\p,t}(\I_n) \leq c_3 \frac{p_n^t}{|T^{\prime}(x)|^{\beta_{\p}(t)}}$ for any $\p$ which satisfies Hypothesis \ref{hyp}, any $t \in I$, any $n \in \N$ and $x \in \I_n$. Therefore for all $\p$ and $t$ we have
\begin{eqnarray*}
\int \log|T^{\prime}|d\mu_{\p,t}& \leq& C\sum_{n \in \mathbb{N}} \sup_{x \in \I_n}\frac{\log|T^{\prime}(x)|}{|T^{\prime}(x)|^{\beta_{\p}(t)}}\leq c_3 \sum_{n \in \mathbb{N}}\frac{2\log n}{n^{2\beta_{\p}(t)}}.
\end{eqnarray*}
Since by (\ref{916}) we know $\beta_{\p}(t) \geq \frac{9}{16}$, the result follows.
\end{proof}

\subsection{Proof of Lemma \ref{top bound}}

Suppose $\p$ satisfies that $p_1, p_2 \geq \epsilon$. By Lemma \ref{lemma periodic} there exists $z \in \{z_{1}, z_{2}, z_{12}\}$ such that $\frac{1}{2}|S_2 \tilde{f}_{\p,t}(z)| \geq c_1$.  Let $z=\Pi(\i)$. Fix $m$ sufficiently large that $\alpha^m \leq \frac{c_1}{2c_2}$. By Lemma \ref{measure lemma} and the fact that $\inf_{x \in \I_1 \cup \I_2} \frac{1}{|T^{\prime}(x)|}=\frac{1}{9}$,
$$\mu_{\p,t}(\I_{i_1 \ldots i_m}) \geq c_3^{-1} \frac{\epsilon^{\frac{m}{4}}}{9^m}.$$

By (\ref{beta2}), (\ref{strategy}), Corollary \ref{rewrite2} and Lemma \ref{e1e2}
\begin{eqnarray}
\beta_{\p}^{\prime\prime}(t) \geq  \frac{c_1^2 \epsilon^{\frac{m}{4}}}{4 \cdot 9^mc_3L}.
\label{gamma}
\end{eqnarray}
Similarly if instead $\p$ satisfies that $p_1> \psi$, by Lemma \ref{lemma periodic} it follows that $f_{\p,t}(z_1) \geq c_1$. Let $z_1=\Pi(\i)$. Again by Lemma \ref{measure lemma},
$$\mu_{\p,t}(\I_{i_1 \ldots i_m}) \geq c_3^{-1} \frac{\psi^{\frac{m}{4}}}{9^m} \geq \frac{\epsilon^{\frac{m}{4}}}{9^m}$$
and therefore (\ref{gamma}) also holds. Set $\gamma_{\epsilon}= \frac{\epsilon^{\frac{m}{4}}}{9^m}$. Then 
\begin{eqnarray*}
-1=\beta_{\p}(1)-\beta_{\p}(0)= \int_0^1 \beta_{\p}^{\prime}(t) \textup{d}t \leq \beta_{\p}^{\prime}(1)- r \gamma_{\epsilon}
\end{eqnarray*}
for some constant $r$. The result follows from the fact that $-\beta_{\p}^{\prime}(1)=\dim \mu_{\p}$. \qed

\section{Proof of Theorem \ref{main}}

By Lemma \ref{infinite thm} we can restrict to the case where $h(\mu_{\p})< \infty$. The proof then follows from Lemmas \ref{tail lemma} and \ref{top bound}. Fix $\epsilon=\epsilon_0>0$ that satisfies Lemma \ref{tail lemma}. If $\p$ does not satisfy Hypothesis \ref{hyp} then either $\dim \mu_{\p} \leq \frac{3}{4}$ or 
$$\dim \mu_{\p} \leq s+\frac{\kappa(s)}{\lambda_0}<1$$
by Lemma \ref{tail lemma}. Otherwise, by Lemma \ref{top bound} 
$$\dim \mu_{\p} \leq 1-r\gamma_{\epsilon_0}$$
which proves the existence of a dimension gap.

\section{Generalisations}

The method used in this paper can be generalised to prove the existence (and bounds on) a dimension gap for more general countable branch expanding maps under a suitable `non-linearity' assumption on the map.

In particular, let $\{\mathcal{I}_n\}_{n \in \mathbb{N}}$ be a countable collection of open non-empty disjoint subintervals of $[0,1]$ such that $(0,1) \subset \bigcup_{n \in \N} \overline{\I_n}$ and let $T_n: \overline{\I}_n \to [0,1]$ be a sequence of expanding bijective $C^2$ maps (so $|T_n^{\prime}|>1$). Define $T: [0,1] \to [0,1]$ as
\begin{eqnarray*}
\begin{array}{ccccc}
T(x)&=& T_n(x) & \textnormal{if}& x \in \overline{\I}_n \\
T(0)&=&0 & & 
\end{array}
\end{eqnarray*}
where we put $T(x)=T_k(x)$ for $k= \min\{n: x \in \overline{\mathcal{I}}_n\}$ if $x$ is a common endpoint of two intervals. Similarly, we adopt the convention that $T^{\prime}(x)= T_k^{\prime}(x)$ where $k=\min\{n: x \in \overline{\mathcal{I}}_n\}$. 

Let $T:[0,1] \to [0,1]$ be a countable branch expanding (Markov) map as described above. Additionally assume that $T$ satisfies the following conditions:
\begin{enumerate}
\item \textbf{Some iterate of $T$ is uniformly expanding.} There exists $l \in \N$ and $\Lambda>1$ for which
$$|(T^l)^{\prime}(x)| \geq \Lambda >1$$ 
for all $x \in [0,1]$.
\item \textbf{R\'enyi condition.} 
There exists $\kappa<\infty$ such that
\begin{eqnarray}
\sup_{n \in \mathbb{N}} \sup_{x, y, z \in \I_n} \left| \frac{T^{\prime\prime}(x)}{T^{\prime}(y)T^{\prime}(z)} \right| = \kappa < \infty. \label{renyi}
\end{eqnarray}
\item \textbf{Fast decaying interval lengths.} There exists $s<1$ such that
$$\sum_{n \in \N} |\I_n|^s < \infty.$$
\item \textbf{Non-linearity assumption.} 
\begin{eqnarray}
T^{\prime}(z_1)T^{\prime}(z_2) \neq T^{\prime}(z_{12})T^{\prime}(z_{21}).
\label{nonlin}
\end{eqnarray}
\end{enumerate}
Then there exists some $\eta>0$ for which
$$\sup_{\p \in \mathcal{P}} \dim \mu_{\p} \leq 1-\eta$$
and $\eta$ depends on $\Lambda, l, \kappa, s$ and a `non-linearity' constant $\theta$ which is given by 
$$\theta=\left|\log \frac{T^{\prime}(z_1)T^{\prime}(z_2)}{T^{\prime}(z_{12})T^{\prime}(z_{21})}\right| \neq 0. $$

We now make some remarks about assumptions (1)-(4). Firstly, (2) guarantees that $-\log|T^{\prime}|$ is locally H\"older, which we saw was crucial for the proof. This in turn allows one to show that an analogue of Proposition \ref{bd} holds for $T$, which is also utilised at many points throughout the proof. In fact, the reason why our method yields a particularly poor estimate on the dimension gap when $T$ is the Gauss map is precisely because the constant $\kappa$ in (\ref{renyi}) is given by $\kappa=16$, which ends up appearing in several exponents throughout the proof. 

Next, we note that (3) is a sharp condition. To see this, suppose there does not exist $s<1$ for which
$ \sum_{n \in \mathbb{N}} |\mathcal{I}_n|^s < \infty.$
Let $0<t<1$ be arbitrary. By assumption
$\sum_{n=1}^{\infty} |\I_n|^t= \infty.$
Thus, we can choose some large $N$ for which
$\sum_{n=N}^{k} |\I_n|^t \geq 1$
for some $k>N$. Fix $\mathbf{p}_N= (p_1, p_2, \ldots)$ where 
\begin{eqnarray*}
p_n=
\left\{ \begin{array}{cccc} 0 & n < N &\textnormal{or}& n > k \\
c|\I_n|^t & N \leq p_n \leq k & & \\
\end{array}
\right.
\end{eqnarray*}
where $c$ is a normalising constant so that $\sum_{n=N}^{k}c|p_n|^t=1$. Consider the Bernoulli measure $\mu_{\p_N}$. Since $h(\mu_{\p_N}) < \infty$ it follows that the dimension $\dim \mu_{\p_N}= \frac{h(\mu_{\p_N})}{\chi(\mu_{\p_N})}$. Applying the analogue of Proposition \ref{bd} it follows that there exists a uniform constant $C>0$ for which $\log |T^{\prime}(x)| \leq -\log|\I_n|+ C$ for all $n \in \N$ and all $x \in \I_n$. Therefore
\begin{eqnarray*}
\dim \mu_{\p_N} &\geq& \frac{- \sum_{n=N}^{k} c|\I_n|^t \log c|\I_n|^t}{-\sum_{n=N}^{k} c|\I_n|^t(\log|\I_n|-C)} \\
&=& \frac{-t \sum_{n=N}^{k} (c|\I_n|^t \log |\I_n|)- \log c}{-\sum_{n=N}^{k}( c|\I_n|^t\log|\I_n|)+ C}.
\end{eqnarray*}

Since $N$ can be chosen arbitrarily large to make $-\sum_{n=N}^{k}( c|\I_n|^t\log|\I_n|)$ arbitrarily large, we deduce that $\dim \mu_{\p_N} \to t$ as $N \to \infty$. Therefore, for all $0<t<1$ we can choose a Bernoulli measure with dimension greater than $t$, proving that a dimension gap does not exist. 

Finally, (4) describes the fact that $T$ sees some non-linearity on one of the first two branches. This is precisely the property that was used in Lemma \ref{periodic} and would be sufficient (though not necessary) to prove an analogue of Lemma \ref{periodic} for a more general map $T$.

\vspace{0.5cm}
\noindent \textbf{Acknowledgements.}  This paper is part of the author's PhD thesis conducted at the University of Warwick. The author is extremely grateful to her supervisor Mark Pollicott for suggesting the problem which is  studied in this paper and for many useful discussions. This paper was written while the author was supported by a \emph{Leverhulme Trust Research Project Grant} (RF-2016-194). The author would also like to thank Marc Kesseb\"ohmer, Thomas Jordan and Simon Baker for helpful conversations.

\end{document}